\documentclass[11pt,draft]{amsart}
\usepackage{amssymb, amstext, amscd, amsmath, amssymb}
\usepackage{mathtools, xypic, paralist, color, dsfont, rotating}
\usepackage{verbatim}
\usepackage{enumerate}
\usepackage{hyperref}

\numberwithin{equation}{section}

\usepackage[norelsize,linesnumbered,ruled]{algorithm2e}

\makeatletter
\def\@cite#1#2{{\m@th\upshape\bfseries%
[{#1\if@tempswa{\m@th\upshape\mdseries, #2}\fi}]}}
\makeatother
\theoremstyle{plain}
\newtheorem{theorem}{Theorem}[section]
\newtheorem{corollary}[theorem]{Corollary}
\newtheorem{proposition}[theorem]{Proposition}

\theoremstyle{definition}
\newtheorem{definition}[theorem]{Definition}
\newtheorem{example}[theorem]{Example}

\newtheorem{remark}[theorem]{Remark}

\newtheorem*{acknow}{Acknowledgements}
\newtheorem*{algo}{Follower Set Graph Algorithm}
\theoremstyle{remark}

\renewcommand{\qedsymbol}{{\vrule height5pt width5pt depth1pt}}
\mathtoolsset{centercolon}
  \newcommand{\A}{{\mathcal{A}}}
  \newcommand{\B}{{\mathcal{B}}}

  \newcommand{\F}{{\mathcal{F}}}

  \newcommand{\K}{{\mathcal{K}}}
\renewcommand{\L}{{\mathcal{L}}}

  \newcommand{\T}{{\mathcal{T}}}
  \newcommand{\U}{{\mathcal{U}}}

\def\al{\alpha}

\def\ga{\gamma}

\def\ze{\zeta}

\def\La{\Lambda}
\def\om{\omega}
\def\Om{\Omega}
\def\si{\sigma}
\def\Si{\Sigma}

\newcommand\vth{\vartheta}
\newcommand\vpi{\varphi}

\newcommand{\bC}{\mathbb{C}}

\newcommand{\bF}{\mathbb{F}}
\newcommand{\bN}{\mathbb{N}}

\newcommand{\bZ}{\mathbb{Z}}

\newcommand{\fF}{{\mathfrak{F}}}
\newcommand{\fG}{{\mathfrak{G}}}
\newcommand{\fH}{{\mathfrak{H}}}

\newcommand{\fL}{{\mathfrak{L}}}

\newcommand{\foral}{\text{ for all }}
\newcommand{\qand}{\quad\text{and}\quad}

\newcommand{\qiff}{\quad\text{if and only if}\quad}
\newcommand{\qfor}{\quad\text{for}\ }

\newcommand{\ca}{\mathrm{C}^*}

\newcommand{\ol}{\overline}
\newcommand{\wt}{\widetilde}

\newcommand{\alg}{\operatorname{alg}}

\newcommand{\mt}{\emptyset}

\newcommand{\spn}{\operatorname{span}}

\newcommand{\supp}{\operatorname{supp}}

\addtocontents{toc}{\protect\setcounter{tocdepth}{1}}

\begin{document}


\title[On the Quantized Dynamics of Languages]{On the Quantized Dynamics of Factorial Languages}

\author[C. Barrett]{Christopher Barrett}
\address{School of Mathematics and Statistics\\ Newcastle University\\ Newcastle upon Tyne\\ NE1 7RU\\ UK}
\email{c.barrett2@newcastle.ac.uk}

\author[E.T.A. Kakariadis]{Evgenios T.A. Kakariadis}
\address{School of Mathematics and Statistics\\ Newcastle University\\ Newcastle upon Tyne\\ NE1 7RU\\ UK}
\email{evgenios.kakariadis@ncl.ac.uk}



\maketitle

\begin{abstract}
We study local piecewise conjugacy of the quantized dynamics arising from factorial languages.
We show that it induces a bijection between allowable words of same length and thus it preserves entropy.
In the case of sofic factorial languages we prove that local piecewise conjugacy translates to unlabeled graph isomorphism of the follower set graphs.
Moreover it induces an unlabeled graph isomorphism between the Fischer covers of irreducible subshifts.
We verify that local piecewise conjugacy does not preserve finite type nor irreducibility; but it preserves soficity.
Moreover it implies identification (up to a permutation) for factorial languages of type $1$ if, and only if, the follower set function is one-to-one on the symbol set.
\end{abstract}

\section{Introduction}

The fruitful interplay between Symbolic Dynamics and Operator Algebras was established in the seminal paper of Cuntz-Krieger \cite{CK80}.
Following their work, Matsumoto introduced an effective way for associating operators to subshifts and forming Cuntz-Krieger-type C*-algebras \cite{Mat97} that were further examined in-depth in a series of papers \cite{Mat98, Mat99, Mat02}.
This theory was revisited with Carlsen \cite{Car08, CM04} and a new view was exploited in more generality in \cite{Mat02}.
These important works clarified a strong connection between intrinsic properties of subshifts with related C*-algebras.

Matsumoto operators follow from a Fock representation that accommodates more structures.
Shalit-Solel \cite{SS09} provided such a context for homogeneous ideals in general and established a rigidity programme for the related (nonselfadjoint) tensor algebras.
The origins of this framework are traced back to the seminal work of Arveson \cite{Arv67}.
Since then, a number of rigidity results have appeared in the literature for tensor algebras of graphs or dynamical systems as in the work of Katsoulis-Kribs \cite{KK04} and Solel \cite{Sol04},  Davidson-Katsoulis \cite{DK08, DK11} that supersedes the work of previous authors \cite{AJ69, HH77, Pet84, Pow92}, Davidson-Roydor \cite{DR11}, Davidson-Ramsey-Shalit \cite{DRS11, DRS15}, Dor-On \cite{Dor15}, Katsoulis-Ramsey \cite{KR16}, and the work of the second author with Davidson \cite{DK12} and Katsoulis \cite{KK14}.
In this endeavour Davidson-Katsoulis \cite{DK11} developed the notion of piecewise conjugacy for classical systems as the essential level of equivalence obtained from tensor algebras. 
Piecewise conjugacy allows for comparisons of the systems locally and thus is more tractable than (global) conjugacy.
Recently it found significant applications to Number Theory and reconstruction of graphs as exhibited in the work of Cornelissen-Marcolli \cite{CM11, CM13} and Cornelissen \cite{Cor12}.

Along this line of research, the second author with Shalit examined tensor algebras of factorial languages in \cite{KS15}.
A factorial language $\La^*$ on $d$ symbols is a subset of the free semigroup $\bF_+^d$ such that if $\mu \in \La^*$ then every subword of $\mu$ is also in $\La^*$.
To fix notation, the operators $T_\mu$ in discussion act on $\ell^2(\La^*)$ and are defined by
\[
T_\mu e_\nu
:=
\begin{cases}
e_{\mu \nu} & \text{ if } \mu \nu \in \La^*,\\
0 & \text{ otherwise},
\end{cases}
\]
for $\mu \in \La^*$.
As an intermediate step we use the C*-algebra of \emph{checkers}
\[
A := \ca(T_\mu^* T_\mu \mid \mu \in \La^*).
\]
The $*$-endomorphisms on $A$ given by $\al_i(a) := T_i^* a T_i$ play an important role in the analysis and the system $(A,\al_1, \dots, \al_d)$ was coined in \cite{KS15} as \emph{the quantized dynamics of $\La^*$}.
Two norm-closed subalgebras of $\B(\ell^2(\La^*))$ can be related to the same language $\La^*$:
\begin{enumerate}[\, I.]
\item The \emph{$\A$-tensor algebra} $\A_{\La^*} := \ol{\alg}\{I, T_\mu \mid \mu \in \La^*\} $ in the sense of Shalit-Solel \cite{SS09};
\item The \emph{$\T^+$-tensor algebra} $\T_{\La^*}^+ := \ol{\alg}\{a, T_\mu \mid a \in A, \mu \in \La^*\}$ in the sense of Muhly-Solel \cite{MS98}.
\end{enumerate}
Matsumoto's C*-algebra is the quotient of $\ca(T):= \ca(T_\mu \mid \mu \in \La^*)$ by the compacts $\K$ in $\ell^2(\La^*)$.

Arveson's Programme on the C*-envelope\footnote{
\ Arveson's Programme was initiated in \cite{Arv69} and established in \cite{Ham79}.
See also its formulation in \cite{KP13}.}
provides a solid pathway for researching possible Cuntz-Krieger-type C*-algebras.
For example the natural analogues related to C*-correspondences are exactly the C*-envelopes of the tensor algebras as proven by Katsoulis-Kribs \cite{KK06}.
One of the main results in \cite{KS15} states that the C*-algebra that fits Arveson's Programme for both $\A_{\La^*}$ and $\T_{\La^*}^+$ is the quotient of $\ca(T)$ by the generalized compacts; rather than quotienting by all compacts as is done in Matsumoto's work.
In fact the quantized dynamics trigger a dichotomy: the C*-envelope of both $\A_{\La^*}$ and $\T_{\La^*}^+$ is either the quotient by all compacts or it coincides with $\ca(T)$, depending on whether the quantized dynamics is injective or not.
This is in full analogy to what holds for graph C*-algebras where sinks or vertices emitting infinite edges are excluded from the Cuntz-Krieger relations.

Apart from being a starting point for Cuntz-Krieger-type C*-algebras \emph{via} the C*-envelope machinery, both $\A_{\La^*}$ and $\T_{\La^*}^+$ are rigid for factorial languages.
It is shown in \cite{KS15} that they encode the factorial language, yet in two essentially different ways:
\begin{enumerate}[\, (i)]
\item The $\A$-tensor algebras provide a complete invariant for the factorial languages up to a permutation of the symbols.
\item The $\T^+$-tensor algebras provide a complete invariant for local piecewise conjugacy (l.p.c.) of the quantized dynamics.
\end{enumerate}
However it was left open how l.p.c. reflects the initial data:
\begin{enumerate}[$\quad$ (a)]
\item[(a)] How is l.p.c. interpreted in terms of factorial languages?
\item[(b)] What properties are (thus) preserved under l.p.c.?
\item[(c)] What is the impact on sofic factorial languages?
\end{enumerate}
In the current paper we answer these questions that add on the impact of the rigidity results of \cite{KS15}.

Before we move to the description of our results we stress that languages of subshifts form special examples of factorial languages and several constructions apply to this broader context.
Thus terminology related to subshifts is extended accordingly to cover general factorial languages, when possible.
Unlike to Carlsen \cite{Car08}, Matsumoto \cite{Mat97} or Krieger \cite{Kri84}, our study is based on the allowable words rather than the points of an induced subshift.
In fact the dynamical system of the backward shift is not explicitly used for the Fock space quantization and thus no connections between l.p.c. and topological conjugacy arise.
The results and examples herein show that they are incomparable.
On one hand l.p.c. requires for the languages to have the same number of symbols (Definition \ref{D: QPloc}) and so it is not implied by topological conjugacy.
On the other hand in Example \ref{E: even vs sft} we construct a subshift of finite type that is l.p.c. to the even shift, and so l.p.c. does not imply topological conjugacy.

We begin by giving an updated picture of the quantized dynamics (Section \ref{Ss: qd}).
We then show that l.p.c. implies a bijection between allowable words of the same length, and thus it respects entropy (Proposition \ref{P: same size}, Corollary \ref{C: entropy}).
The flexibility of l.p.c. can be seen in the form of this bijection; but has its limitations (see Remark \ref{R: fn}).
Most notably l.p.c. does not preserve finite type as shown in Example \ref{E: even vs sft}, nor irreducibility as shown in Example \ref{E: even vs sft irr}.
Consequently l.p.c. does not preserve the zeta function (Remark \ref{R:zeta fun 2sided}).

Nevertheless l.p.c. respects soficity (i.e. the C*-algebra of checkers is finite dimensional) where the theory is rich.
There is a well known construction of a labeled finite graph, i.e. \emph{the follower set graph}, that gives a presentation of a sofic factorial language.
When the language is of finite type then this construction can be described by a (terminating) algorithm (Follower Set Graph Algorithm).
Labeled graph isomorphism is equivalent then with the factorial languages being the same up to a permutation of symbols (Proposition \ref{P: 1 block code}).
On the other hand it is the unlabeled graph isomorphism that coincides with l.p.c. (Theorem \ref{T:unlbl iso}).
Combining these results with \cite{KS15} we thus prove the following diagram for two sofic factorial languages $\La^*$ and $M^*$:
\[
\xymatrix@C=20pt@R=20pt{
\A_{\La^*} \simeq \A_{M^*} \ar@{=>}[d] \ar@{<=>}[r] & \text{the labeled f.s.g. of $\La^*$ and $M^*$ are isomorphic.} \ar@{=>}[d] \\
\T_{\La^*}^+ \simeq \T_{M^*}^+ \ar@{<=>}[r] & \text{the unlabeled f.s.g. of $\La^*$ and $M^*$ are isomorphic.}
}
\]

The follower set graph construction is rather useful in Theoretical Computer Science as the starting point for computing minimal presentations.
Such presentations are unique (up to isomorphism) for irreducible two-sided subshifts and are better known as Fischer covers \cite{Fis75-1, Fis75-2}.
We show that l.p.c. induces an unlabeled graph isomorphism between the Fischer covers of irreducible two-sided subshifts (Corollary \ref{C: irr}).
It is quite interesting to notice though that we achieve these results without inducing a bijection between intrinsically synchronizing words.
A weaker bijection between the collections of follower sets of such words is induced.
This is depicted in Remark \ref{R: intr mix} where we show the limitations of our arguments.
The same obstructions do not allow to apply our arguments and prove (or disprove) that l.p.c. respects the mixing property.

We further investigate cases where the vertical directions in the diagram above can or cannot be equivalences based on the type and the number of symbols.
As commented in \cite{KS15} these arrows cannot be reversed in general and we extend this remark for two-sided subshifts in Example \ref{E:not same sft}.
The key in these counterexamples is that the follower set function is not one-to-one.
Apparently this is the only obstruction for type $1$ factorial languages.
In Theorem \ref{T: same graph iso} we show that unlabeled isomorphism for type $1$ factorial languages produces a labeled isomorphism if and only if the follower set function is one-to-one on the symbol set.
Consequently then, isomorphism of the $\A$-tensor algebras is equivalent to isomorphism of the $\T^+$-tensor algebras.
This condition is satisfied by edge shifts with invertible adjacency matrix.
Injectivity of the follower set function is not required for type 1 factorial languages on two symbols.
In Remark \ref{R: sft 1 symb 2} we describe how an unlabeled graph isomorphism implies a labeled graph isomorphism of the follower set graphs in these cases.
These results depend on the low complexity of the system.
However this does not hold when passing to type $2$ factorial languages, even when the number of symbols is small (Examples \ref{E:counter1} and \ref{E:counter2}).

To facilitate comparisons, we developed a program that takes as an input a set of forbidden words on two symbols and gives the follower set graph as an output\footnote{\
We chose to develop this program for the right version of factorial languages, rather than the left we work with here, as it is accustomed in Theoretical Computer Science to concatenate on the right.
Nevertheless, the left version follows easily by reversing the words in the input and the output.}.
The code and the .exe file can be downloaded from the official webpage of the second author, currently at
\begin{center}
\href{http://www.mas.ncl.ac.uk/\~nek29}{\texttt{http://www.mas.ncl.ac.uk/$\sim$nek29}}.
\end{center}
We remark that our principal objective here was to construct a program that computes easily follower set graphs as a check for our examples and counterexamples.
Hence we were not (extremely) concerned about complexity or the required memory, but only about the fact that it terminates when the type is finite.

\begin{acknow}
This project started as a continuation of \cite{KS15} with Orr Shalit.
Following his suggestion it was decided for the paper to go with just two authors.
We thank Orr for the numerous comments and corrections on earlier drafts of the paper.

We also thank Vissarion Fisikopoulos for his valuable feedback and corrections on earlier versions of the algorithm and the program that computes the follower set graphs.
A friend in need is a friend indeed.

This paper is part of the MMath project of the first author.
\end{acknow}

\section{Preliminaries}

\subsection{Languages and subshifts}\label{Ss: subshifts}

Let us fix the terminology and notation we will be using throughout the paper.
To this end $\bF_+^d$ denotes the free semigroup on the symbol set $\Sigma :=\{1, \dots, d\}$ with multiplication given by concatenation.
For $\mu = \mu_k \dots \mu_1$ in $\bF_+^d$ we write $|\mu| := k$ for the length of $\mu$.
The empty word $\mt$ is by default in $\bF_+^d$, it has zero length and plays the role of the unit.

For $\mu, \nu \in \bF_+^d$ we say that $\nu$ is a \emph{subword} of $\mu$ if there are $w,z \in \bF_+^d$ such that $\mu = w \nu z$.
By default the empty word is a subword of every $\mu \in \La^*$.
A \emph{(factorial) language} is a subset $\La^*$ of $\bF^d_+$ that satisfies the following property:
\begin{center}
if $\mu \in \La^*$ then every subword of $\mu$ is in $\La^*$.
\end{center}
We will simply use the term ``language'' since we are going to encounter just factorial ones.
Without loss of generality we will always assume that all letters of the symbol set are in $\La^*$.
Otherwise we view $\La^*$ to be defined on less symbols, i.e. on $\Si$ out the symbols that are not in $\La^*$.

Examples of languages arise from subshifts and below we give a brief description.
Several elements from the theory of subshifts apply to the broader context of languages and terminology will be extended to cover languages in general.
Apart from two-sided we also consider one-sided subshifts.
Several results that hold for the two sided version hold also for the one-sided with almost the same proof.
We will mainly discuss left subshifts, but similar comments hold for the right subshifts.
In order to make sense of the one-sided subshifts and avoid technicalities we make the following convention.
We will write the sequences $x = (x_n) \in \Si^{\bZ_+}$ from right to left, i.e.
\[
\dots x_n \dots x_1 x_0. = x
\]
and likewise for elements in $\Sigma^{\bZ}$.
This is to comply with operator composition which comes by multiplying on the left.

We endow $\Si^{\bZ_+}$ with the product topology and we fix $\si \colon \Si^{\bZ_+} \to \Si^{\bZ_+}$ be the backward shift with $\si((x_i))_k = x_{k+1}$.
With our convention, the map $\si$ shifts to the right.
The pair $(\La, \si)$ is called a \emph{left subshift} if $\La$ is a closed subset of $\Si^{\bZ_+}$ with $\si(\La) \subseteq \La$.
Similarly the pair $(\La, \si)$ is called a \emph{two-sided subshift} if $\La$ is a closed subset of $\Si^{\bZ}$ with $\si(\La) = \La$.

We write $x_{[m+n-1, m]}$ for the block $x_{m+n-1} \dots x_m$ in $x \in \La$.
A word $\mu = \mu_{|\mu|} \dots \mu_1$ is said to \emph{occur} in some (one-sided or two-sided) sequence $x$ if there is an $m$ such that
\[
x_{|\mu| - 1 + m} = \mu_{|\mu|}, \dots, x_{m} = \mu_1.
\]
If a word occurs in some point of $\La$ then it is called \emph{allowable}.
The \emph{language of a subshift $\La$} is defined by
\[
\La^* := \{w \in \bF_+^d \mid w \text{ occurs in some } x \in \La\}.
\]
Since $\La$ is $\si$-invariant we have that for every allowable word $\mu$ there exists an $x \in \La$ such that $x_{[|\mu|, 0]} = \mu$.
We write $\B_n(\La^*)$ for the allowable words of length $n$ in $\La^*$.
By following the same arguments as in the two-sided subshifts one can show that if $\La$ defines a left subshift then $\La^*$ is a language such that:
\begin{center}
for every $\mu \in \La^*$ there is a $\mt \neq \nu \in \La^*$ such that $\nu \mu \in \La^*$.
\end{center}
Conversely every language with this property defines uniquely a left subshift (see \cite[Proposition 12.3]{SS09} and \cite{KS15}).

Subshifts can be described also in terms of forbidden words.
Let $\fF$ be a set of words on the symbol set $\Si = \{1, \dots, d\}$, and let
\[
\La_\fF : = \{(x_n) \in \Si^{\bZ_+} \mid \, \text{no $\mu \in \fF$ occurs in $(x_n)$} \, \}.
\]
It is known that all two-sided subshifts arise in this way.
Likewise this also holds for one-sided subshifts.
By setting
\[
\fF_k := \{ \mu \in \fF \mid \, \text{$\mu$ does not occur in $(x_n) \in \La$}, |\mu| \leq k \, \}
\]
we see that $\fF_k \subseteq \fF_{k+1}$ and $\fF = \bigcup_k \fF_k$.
Then we have that $\La = \cap_k \La_{\fF_k}$, where the intersection is considered inside the full shift space on $\Si$.
The elements in $\fF$ are called \emph{the forbidden words of the subshift}.
We will call a forbidden word \emph{minimal}, if all of its proper subwords are allowable.

Every set $\fF$ of forbidden words admits a unique basis $F \subseteq \fF$ in the sense that for every $\mu \in \fF$ there are (unique) $\nu, w \in \bF_+^d$ and a $\mu' \in F$ such that $\mu = \nu \mu' w$ and $\mu'$ is minimal.
We say that $\La$ is a \emph{subshift of finite type} (SFT) if the longest word in the basis of $\fF$ has finite length.
We say that an SFT is of \emph{type $k$} if the longest word in the basis of $\fF$ has length $k + 1$.
Hence if $\La$ is of type $k$ then any forbidden word of length strictly greater than $k + 1$ cannot be minimal.

The notions of forbidden words, minimality, basis and type pass naturally to any language.
For example, given a set $\fF$ we can define a language by 
\[
\La^*_{\fF} = \bF_+^d \setminus \{w \mu \nu \mid \mu \in \fF, w, \nu \in \bF_+^d\}.
\]
Not every language is a language of a subshift but it can be embedded in one by augmenting the symbol set.
Suppose that $\La^*$ is defined through a set of forbidden words $\fF$ in $\Si$.
We introduce a new distinguished symbol $\ze$ and take the symbol set $\wt{\Si} = \{1, \dots, d, \ze\}$.
Then \emph{the augmented subshift $(\wt{\La}, \si)$ of $\La^*$} is the subshift defined by $\fF$ in $\wt{\Si}^{\bZ}$.
Since $\ze$ is not contained in any word in $\fF$ it follows that the language of $\wt{\La}$ is 
\[
\wt{\La}^* = \{\ze^{n_k} \mu_k \dots \mu_2 \ze^{n_2} \mu_1 \ze^{n_1} \mid n_1, \dots, n_k \in \bZ_+, \mu_1, \dots, \mu_k \in \La^*\}
\]
and thus contains $\La^*$.

Recall that a two-sided subshift $(\La,\si)$ is called \emph{sofic} if the number of classes in $\La^*$ with respect to the equivalence relation
\[
\mu \sim \nu \Leftrightarrow \{w \in \La^* \mid w \mu \in \La^*\} = \{w \in \La^* \mid w \nu \in \La^*\}
\]
is finite.
Equivalently, if $(\La, \si)$ is a factor of an SFT \cite{Fis75-1, Wei73}.
The reader is addressed for example to \cite[Theorem 3.2.10]{LM95} for a modern treatment of sofic subshifts.
For languages that do not come from subshifts we will use the definition of soficity in terms of the equivalence classes.
It is shown in \cite{KS15} that a language $\La^*$ is sofic (resp. of finite type) if and only if its augmented subshift $(\wt{\La},\si)$ is sofic (resp. of finite type).
This follows by observing that $\mt \sim \ze \mu$ for every $\mu \in \wt{\La}^*$.

Every two-sided subshift becomes a compact metric space.
Taking the one-sided subshifts to be closed yields the same result in our case.
Therefore every sequence in a subshift has a converging subsequence.
This often appears in \cite{LM95} as the Cantor's diagonal argument, mainly because metric spaces come later in the presentation of \cite{LM95}.
We preserve this terminology to keep connections with Symbolic Dynamics.
However it is interesting that this argument works to build the one-sided subshift from a set of predetermined forbidden words.
The key is that the one-sided full shift is compact and metrizable with the topology given by
\[
\rho(x,y) =
\begin{cases}
2 & \text{ if $x_0 \neq y_0$,}\\
2^{-k} & \text{ if $x \neq y$ and $k$ is maximal so that $x_{[k,0]} = y_{[k,0]}$,}\\
0 & \text{ if $x = y$}.
\end{cases}
\]

\subsection{Fock representation}

We will require some basic theory from Hilbert spaces to show how the quantized dynamics arise from a language.
The reader who is not familiar with operator theory may read this subsection in combination with Section \ref{Ss: qd} where explicit identifications in terms of topological spaces are provided.

Operator algebras associated to subshifts were introduced by Matsumoto \cite{Mat97}.
Let $\La^*$ be a language on $d$ symbols.
Let $H = \ell^2(\La^*)$ and fix the operators $T_i$ such that $T_i e_\mu = e_{i\mu}$ if $i\mu \in \La^*$ or zero otherwise.
We fix
\[
\ca(T):=\ca(I, T_i \mid i=1, \dots, d).
\]
It is convenient to write $T_\mu T_\nu = T_{\mu\nu}$ even when $T_{\mu\nu} = 0$, i.e.  $T_{\mu\nu} = 0$ if and only if $\mu \nu \notin \La^*$.
We will also write $T_\mt = I$.
Likewise we write $e_{\mu \nu} = 0$ in $\ell^2(\La^*)$ when $\mu\nu \notin \La^*$.
The operators $T_\mu$ satisfy a list of properties:
\begin{enumerate}
\item $T_\mu^*T_\mu$ is an orthogonal projection on $\ol{\spn}\{e_\nu \mid \mu\nu \in\La^*\}$.
\item $T_\nu T_\nu^*$ is an orthogonal projection on $\ol{\spn}\{e_{\nu\mu} \mid \mu \in\La^*\}$.
\item If $|\mu| = |\nu|$, then $T_\mu^* T_\nu = 0$ if and only if $\mu \neq \nu$.
\item $T_\mu^*T_\mu$ commutes with $T_\nu^* T_\nu$, and with $T_\nu T_\nu^*$.
\item $T_\mu^* T_\mu \cdot T_i = T_i \cdot T_{\mu i}^* T_{\mu i}$ for all $i = 1, \dots, d$.
\item $\sum_{i=1}^d T_i T_i^* + P_\mt = I$ where $P_\mt$ is the projection on $\bC e_\mt$.
\item The rank one operator $e_\nu \mapsto e_\mu$ equals $T_\mu P_\mt T_\nu^*$.
\item The ideal $\K(\ell^2(\La^*))$ of compact operators is in $\ca(T)$.
\end{enumerate}

In \cite{KS15} the second author with Shalit examine several operator algebras related to the operators $T_i$.
Among them there are two classes of nonselfadjoint operator algebras:
\begin{enumerate}
\item The tensor algebra $\A_{\La^*}$ in the sense of Shalit-Solel \cite{SS09} is defined as the norm-closed subalgebra of $\B(\ell^2(\La^*))$ generated by $I$ and the $T_i$ for $i=1, \dots, d$.
\item The tensor algebra $\T^+_{\La^*}$ in the sense of Muhly-Solel \cite{MS98} is defined as the norm-closed subalgebra of $\B(\ell^2(\La^*))$ generated by $I$, the $T_i$ for $i=1, \dots, d$, and the $T_\mu^* T_\mu$ for $\mu \in \La^*$.
\end{enumerate}
The relations above imply that
\[
\A_{\La^*} = \ol{\spn}\{T_\mu \mid \mu \in \La^*\}
\qand
\T^+_{\La^*} = \ol{\spn}\{T_\mu a \mid \mu \in \La^*, a \in A\}
\]
for the unital C*-subalgebra $A:=\ca(T_\mu^* T_\mu \mid \mu \in \La^*)$ of $\ca(T)$.

\subsection{$Q$-projections}

We will be using the projections generated by the $T_i^*T_i$.
To this end we introduce the following enumeration.
Write all numbers from $0$ to $2^d-1$ by using $2$ as a base, but in reverse order.
Hence we write $[m]_2 \equiv [m] = [m_1 m_2 \dots m_d]$ so that $2=[0100\dots 0]$.
Let the (not necessarily one-to-one) assignment
\[
[m_1 \dots m_d] \mapsto Q_{[m_1 \dots m_d]} : = \prod_{m_i = 1} T_{i}^* T_{i} \cdot \prod_{m_i = 0} (I-T_{i}^*T_{i}),
\]
where $I \in \B(\ell^2(\La^*))$.
For example we write
\[
Q_{0} = Q_{[0\dots 0]} = \prod_{i=1}^d (I - T_i^*T_i) \qand Q_{2^d-1} = Q_{[1\dots 1]} = \prod_{i=1}^d T_i^*T_i.
\]
The $Q_{[m]}$ are the minimal projections in the C*-subalgebra $\ca(I, T_i^*T_i \mid i=1, \dots, d)$ of $\ca(T)$.
Consequently we obtain $\sum_{[m]=0}^{2^d -1} Q_{[m]} = I$.
The second author with Shalit have shown in \cite{KS15} the stronger equality (notice that we sum for $[m] \geq 1$ here)
\[
T_\mu^* T_\mu = T_\mu^* T_\mu \cdot \sum_{[m] = 1}^{2^d-1} Q_{[m]} = \sum_{[m] = 1}^{2^d-1} Q_{[m]} \cdot T_\mu^* T_\mu
\]
for all $\mt \neq \mu \in \La^*$.
It follows that the unit of $\ca(T_i^*T_i \mid i=1, \dots, d)$ coincides with $I \in \B(\ell^2(\La^*))$ if and only if $\La^*$ induces a left subshift \cite[Lemma 4.4]{KS15}.

\section{The quantized dynamics on the allowable words}\label{Ss: qd}

Let $\La^*$ be a language on $d$ symbols.
We fix once and for all the unital C*-subalgebra
\[
A:=\ca(T_\mu^* T_\mu \mid \mu \in \La^*)
\]
of $\ca(T)$.
Then $A = \ol{\cup_l A_l}$ is a unital commutative AF algebra for
\[
A_l := \ca(T_\mu^*T_\mu \mid \mu \in \B_l(\La^*)).
\]
The C*-algebra $A$ can be characterized by using $\La^*$.
For $l \geq 0$ let $\sim_l$ be the equivalence relation on $\La^*$ given by the rule
\[
\mu \sim_l \nu \Leftrightarrow \{w \in \B_l(\La^*) \mid w\mu \in \La^*\} = \{w \in \B_l(\La^*) \mid w \nu \in \La^* \}.
\]
Let the discrete space $\Om_l = \La^*/ \sim_l$ and write $[\mu]_l$ for the points in $\Om_l$.
Every $\mu \in \La^*$ splits $\B_l(\La^*)$ into the set of the $w_i \in \B_l(\La^*)$ for which $w_i \mu \in \La^*$ and its complement.
There is a finite number of such splittings since $\B_l(\La^*)$ is finite.
They completely identify single points in $\Om_l$, and hence $\Om_l$ is a (discrete) finite space.
Furthermore the mapping
\[
\vth \colon \Om_{l+1} \to \Om_l: [\mu]_{l+1} \mapsto [\mu]_l
\]
is a well defined (continuous) and onto map.
We can then form the projective limit $\Om$ by the directed sequence
\[
\xymatrix{
\Om_0 & \Om_1 \ar[l]_\vth & \Om_2 \ar[l]_\vth & \dots \ar[l]_\vth & \ar[l]  \Om
}
\]
for which we obtain the following identification.
In \cite{KS15} the second author with Shalit have shown that
\[
A_l \simeq C(\Om_l) \qand A \simeq C(\Om).
\]
We write $\Om_{\La^*}$ for $\Om$ when we want to highlight the language $\La^*$ to which $\Om$ is related.
From now on we will tacitly identify $A_l$ with $C(\Om_l)$ and $A$ with $C(\Om)$.

In the sequel we will use the same notation $Q_{[m]}$ for the subspaces of $\Om$ that correspond to the projections $Q_{[m]}$ of $\ca(T)$.
To make this precise if $[m] = [m_1 \dots m_d]$ is the binary expansion of a number from $0$ to $2^d - 1$ then the subspace corresponding to $Q_{[m]}$ consists of the points $[\mu] \in \Om$ for which:
\begin{enumerate}
\item $i \mu \in \La^*$ for all $i$ with $m_i =1$; and
\item $i \mu \not\in \La^*$ for all $i$ with $m_i = 0$.
\end{enumerate}
We record here the following proposition from \cite{KS15} for future reference.

\begin{proposition}\label{P: finite A} \cite{KS15}
Let $\La^*$ be a language on $d$ symbols.
Then $\Om$ is finite if and only if $\La^*$ is a sofic language.
If, in particular, $\La^*$ is of finite type $k$ then $A_l = A_{k}$ for all $l \geq k+1$.
\end{proposition}

Therefore if $\La^*$ is sofic then there exists a stabilizing step $k$ for which $A_l = A_k$ for all $l \geq k+1$.
When $\La^*$ is in particular of finite type, then we get the following proposition for the possible equivalent classes.

\begin{proposition}\label{P: classes for sft}
Let $\La^*$ be a language on $d$ symbols.
If $\La^*$ is of finite type $k$ then for all $\mu \in \La^*$ with $|\mu| \geq k$ we have that
\[
[\mu_1 \dots \mu_{|\mu|}]_{k} = [\mu_1 \dots \mu_k]_{k}.
\]
\end{proposition}

\begin{proof}
Fix an allowable word $\mu = \mu_1 \dots \mu_k \dots \mu_{|\mu|}$.
By the properties of the language, if $w\mu \in \La^*$ then $w \mu_1 \dots \mu_k \in \La^*$ as well.
Conversely let $w \in \La^*$ such that $w \mu_1 \dots \mu_k \in \La^*$.
To reach contradiction suppose that $w\mu \notin \La^*$.
Therefore there is an $n_1$ and an $n_2$ such that
\[
\nu := w_{n_1} \dots w_{k} \, \mu_1 \dots \mu_k \, \mu_{k+1} \dots \mu_{n_2} \notin \La^*.
\]
Choose $n_1$ and $n_2$ so that $\nu$ is a minimal forbidden word.
Since $\mu \in \La^*$ (hence $\mu_1 \dots \mu_{n_2} \in \La^*$) we have that $w_{n_1} \dots w_k \neq \mt$; hence $|w| + 1 - n_1 \geq 1$.
Similarly, since $w \mu_1 \dots \mu_k \in \La^*$ (hence $w_{n_1} \dots w_k \mu_1 \dots \mu_k \in \La^*$) we have that $\mu_{k+1} \dots \mu_{n_2} \neq \mt$; hence $n_2 \geq k + 1$.
Therefore $\nu$ is a forbidden word of length at least
\[
(|w| + 1 - n_1) + n_2 \geq k+2.
\]
Since $\La^*$ is of type $k$, the word $\nu$ cannot be minimal, which gives the required contradiction.
\end{proof}

We define the maps $\al_i \colon A \to A$ such that $\al_i(a) = T_i^* a T_i$.
It is clear that every $\al_i$ is a positive map and  takes values in $A$ since
\[
\al_i(T_\mu^* T_\mu) = T_{\mu i}^* T_{\mu i} \in A
\]
for all $\mu \in \La^*$.
In fact every $\al_i$ is a $*$-endomorphism of $A$, since $T_\mu^*T_\mu$ commutes with $T_iT_i^*$ and $T_i T_i^* T_i = T_i$.
The $\al_i$ induce the required covariant relation
\[
a T_i = T_i \al_i(a) \foral a \in A, i=1, \dots, d.
\]
used in the study of $\T_{\La^*}^+$.
We use the identification of $A$ with $C(\Om)$ to get a translation of each $\al_i$ as a continuous map $\vpi_i$ partially defined on $\Om$.

Let $A^i$ be the direct summand $T_i^*T_i A$ of $A$, with unit $T_i^*T_i$.
Then $A^i$ is the direct limit of $A^i_l = A_l \cap A^i$, and the corresponding projective limit $\Om^i$ is determined by the spaces
\[
\Om_l^i := \{[\mu]_l \in \Om_l \mid i \mu \in \La^*\}
\]
and the map
\[
\vth \colon \Om_{l+1}^i \to \Om_l^i : [\mu]_{l+1} \to [\mu]_{l}.
\]
Hence $\al_i \colon A \to A^i$ is a unit preserving map from $A = C(\Om)$ into $A^i := T_i^* T_i A = C(\Om^i)$, and therefore induces a continuous map $\vpi_i : \Om^i \rightarrow \Om$.

\begin{proposition}\label{P: graph sft}
Let $\La^*$ be a language on $d$ symbols.
With the above notation we have that $\al_i|_{A_l} \colon A_l \to A_l^i$ is induced by $\vpi_i \colon \Om_l^i \to \Om_l$ such that $\vpi_i([\mu]_{l+1}) = [i\mu]_{l}$.
\end{proposition}

\begin{proof}
We have to show that $\al_i(f) = f \vpi_i$ for all $f \in A_l = C(\Om_l)$.
It suffices to do so for $f = T_\mu^*T_\mu$ with $\mu \in \B_l(\La^*)$.
To this end we have that
\[
\al_i(T_\mu^* T_\mu) = T_{\mu i}^* T_{\mu i} = \chi_{\A}
\]
for the set $\A = \{[w]_{l+1} \mid \mu i w \in \La^*\}$.
On the other hand we have that $T_\mu^* T_\mu = \chi_\B$ for the set $\B = \{[w]_l \mid \mu w \in \La^*\}$.
Hence we compute
\[
\chi_\B\left( \vpi_i([w]_{l+1}) \right)
=
\chi_\B([iw]_{l})
=
\begin{cases}
1 & \text{ if } \mu i w \in \La^*,\\
0 & \text{ otherwise},
\end{cases}
\]
which shows that $\chi_\B \vpi_i = \chi_\A = \al_i(\chi_\B)$.
\end{proof}

The universal property of the projective limit implies that this information is enough to describe $\vpi_i$.
Indeed we have that
\[
\Om = \{(\om_n)_{n \geq 0} \mid \om_n = [\mu]_n, \mu \in \La^*\}.
\]
Therefore for $\om \in \Om^i$, i.e. $\om_n = [\mu]_n$ such that $i\mu \in \La^*$, we get
\[
\vpi_i(\om) = ([i\mu]_n).
\]
In the particular case when $\La^*$ is sofic let $k$ be the step so that $A_l = A_{k}$ for all $l \geq k+1$.
Then we have that $\Om_{l} = \Om_{k}$ for $l \geq k+1$, and hence
\[
\Om \simeq \{[\mu]_{k} \mid \mu \in \La^*\}.
\]
The projective limit description gives that the induced $\vpi_i$ is then given by $[\mu]_{k} \mapsto [i\mu]_{k}$, for $\mu \in \La^*$ with $i\mu \in \La^*$.

\begin{definition}\label{D: qd}
Let $\La^*$ be a language on $d$ symbols.
We call the $(A,\al) \equiv (A,\al_1, \dots, \al_d)$, or alternatively the $(\Om, \vpi) \equiv (\Om,\vpi_1, \dots, \vpi_d)$, the \emph{quantized dynamics of $\La^*$}.
\end{definition}

\begin{remark}
The intersection of the subspaces corresponding to the $T_\mu^* T_\mu$ is always non-empty since $\prod_{\mu \in \La^*} T_\mu^*T_\mu e_\mt = e_\mt$.
In fact the intersection consists of the single point $([\mt]_n) \in \Om$.
It is evident that $([\mt]_n)$ is in the domain of all $\vpi_\mu$ with $\mu \in \La^*$.
\end{remark}

\section{Local piecewise conjugacy}

Let $\La^*$ and $M^*$ be languages.
Fix their associated quantized dynamics $(\Om_{\La^*}, \vpi)$ and $(\Om_{M^*}, \psi)$, and let $Q$ and $P$ be the corresponding systems of projections.
Recall that a $*$-isomorphism $\ga \colon C(\Om_{M^*}) \to C(\Om_{\La^*})$ induces a homeomorphism $\ga_s \colon \Om_{\La^*} \to \Om_{M^*}$ on the spectra.

We have the following notion of conjugacy for the quantized dynamics.
For convenience we use the notation
\[
\supp [m] := \{i \in \{1, \dots, d\} \mid m_i = 1\}
\]
where $[m] = [m_1 \dots m_d]$ is the binary expansion of $m \in \{0,1, \dots, 2^d-1\}$.

\begin{definition}\label{D: QPloc}
We say that the systems $(\Om_{\La^*}, \vpi)$ and $(\Om_{M^*}, \psi)$ are \emph{$Q$-$P$-locally piecewise conjugate} if:
\begin{enumerate}
\item there exists a homeomorphism $\ga_s \colon \Om_{\La^*} \to \Om_{M^*}$; and
\item for every $x \in Q_{[m]}$ there is a neighbourhood $x \in \U \subseteq Q_{[m]}$ and an $[n] \in \{0,1, \dots, 2^d-1\}$ such that $|\supp [n]| = |\supp [m]|$, $\ga_s(\U) \subseteq P_{[n]}$, and
\[
\ga_s \vpi_i|_{\U} = \psi_{\pi(i)} \ga_s|_{\U},
\]
for a bijection $\pi \colon \supp [m] \to \supp [n]$.
\end{enumerate}
\end{definition}

Equivalently, every $Q_{[m]} \subseteq \Om_{\La^*}$ has an open cover $\{\U_\pi\}_{\pi}$ indexed by the one-to-one correspondences $\pi \colon \supp [m] \to \{1, \dots, d\}$ such that $\ga_s(\U_\pi) \subseteq P_{[n]}$ for all $[n]$ with $\supp [n] = \pi(\supp [m])$, and $\ga_s \vpi_i|_{\U_\pi} = \psi_{\pi(i)} \ga_s|_{\U_\pi}$.
We do not exclude the case where $\U_\pi = \mt$ for some $\pi$.
As an immediate consequence if an $\om \in Q_{[m]}$ is in the intersection of $\U_{\pi}$ with $\U_{\pi'}$ then
\[
\psi_{\pi(i)} \ga_s(\om) = \ga_s \vpi_i(\om) = \psi_{\pi'(i)} \ga_s(\om).
\]
for all $i \in \supp [m]$.

\begin{remark}\label{R: fn}
We can use local piecewise conjugacy to define inductively maps $f_n \colon \B_n(\La^*) \to\B_n(M^*)$.
For every step we start at the point $\om := ([\mt]_n)$ which is in the intersection of all subspaces corresponding to $T_\mu^* T_\mu$ for $\mu \in \La^*$.
By definition there exists a neighborhood $\U_\pi \subseteq Q_{[1\dots 1]}$ containing $\om$ such that
\[
\ga_s \vpi_i(\om) = \psi_{\pi(i)} \ga_s(\om) \foral i=1, \dots, d.
\]
We write $\pi_{\mt, 0}$ for the bijection $\pi$ and set $f_1 = \pi_{\mt, 0}$.
Consequently we obtain 
\[
\ga_s \vpi_{\mu_1}(\om) = \psi_{\pi_{\mt, 0}(\mu_1)} \ga_s(\om).
\]
For $f_2$ let a word $\mu = \mu_2 \mu_1 \in \B_2(\La^*)$.
Since $\mu_2 \mu_1 \in \La^*$ then $\om$ is in the domain of  $\vpi_{\mu_2} \vpi_{\mu_1} = \vpi_{\mu_1 \mu_2}$.
We apply the same argument for $\om_1 = \vpi_{\mu_1}(\om)$ and find a bijection $\pi_{\mu, 1}$ coming from the neighborhood $\U_{\pi_{\mu,1}}$ of $\om_1$ such that
\[
\ga_s \vpi_{\mu_2}(\vpi_{\mu_1}(\om)) = \psi_{\pi_{\mu, 1}(\mu_2)} \ga_s(\vpi_{\mu_1}(\om)) = \psi_{\pi_{\mu, 1}(\mu_2)} \psi_{\pi_{\mt, 0}(\mu_1)} \ga_s(\om).
\]
In particular we get that $\ga_s(\om)$ is in the domain of
\[
\psi_{\pi_{\mu, 1}(\mu_2)} \psi_{\pi_{\mt, 0}(\mu_1)} = \psi_{\pi_{\mu, 1}(\mu_2) \pi_{\mt, 0}(\mu_1)}
\]
and hence the word $\pi_{\mu, 1}(\mu_2) \pi_{\mt, 0}(\mu_1)$ is in $M^*$.
This procedure gives a well defined map $f_2 \colon \B_2(\La^*) \to \B(M^*)$.
The same argument applies for a word $\mu_1$ of length $n$ (instead of just a letter) and a letter $\mu_2$, and gives a map $f_n \colon \B_n(\La^*) \to \B_n(M^*)$ that is defined by
\[
f_n(\mu_n \dots \mu_1) := \pi_{\mu, n - 1}(\mu_n) \dots \pi_{\mt, 0}(\mu_1).
\]
The notation $\pi_{\mu, k} \equiv \pi_{\mu_k}$ denotes that this bijection comes from the neighborhood $\U_{\pi_{\mu_k}}$ of the point $\vpi_{\mu_{k} \dots \mu_1}(\om)$.

Notice that each $f_n$ depends on the point $\om := ([\mt]_n)$.
However it also depends on the orbit of $\om$.
For a word $\mu = \mu_n \dots \mu_1 \in \La^*$ we have to keep track where we are at under $f_{n-1}$, since the $n$th bijection depends on where $\vpi_{\mu_{n-1}} \cdots \vpi_{\mu_1}(\om)$ sits.
It is clear that
\[
f_n(\mu_n \dots \mu_1) = \pi_{\mu, n - 1}(\mu_n) f_{n-1}(\mu_{n-1} \dots \mu_1),
\]
but in general $f_n(\mu_n \dots \mu_1)$ may be different than $f_1(\mu_n) f_{n-1}(\mu_{n-1} \dots \mu_1)$.
We present such a case in Example \ref{E: even vs sft}.
\end{remark}

Local piecewise conjugacy respects certain properties of languages.
This is very pleasing as several invariants for the usual topological conjugacy of subshifts depend on these data.

\begin{proposition}\label{P: same size}
If $\La^*$ and $M^*$ are locally piecewise conjugate languages then the maps $f_n$ defined in Remark \ref{R: fn} are bijections.
Consequently $|\B_n(\La^*)| = |\B_n(M^*)|$ for all $n \in \bN$.
\end{proposition}

\begin{proof}
First we show that the $f_n$ are one-to-one.
To this end suppose that
\[
\pi_{\mu, n - 1}(\mu_n) \dots \pi_{\mt, 0}(\mu_1) = \pi_{\nu, n -1}(\nu_n) \dots \pi_{\mt, 0}(\nu_1).
\]
Then we get that $\pi_{\mu, i -1}(\mu_i) = \pi_{\nu, i-1}(\nu_i)$ for all $i=1, \dots, n$.
In particular we have that $\pi_{\mt, 0}(\mu_1) = \pi_{\mt, 0}(\nu_1)$ and therefore $\mu_1 = \nu_1$.
Consequently we get that $\pi_{\mu, 1} = \pi_{\nu, 1}$, hence $\mu_2 = \nu_2$.
Inductively we have that $\mu_i = \nu_i$ for all $i=1, \dots n$.
Therefore we get $|\B_n(\La^*)| \leq |\B_n(M^*)|$.
By symmetry we obtain equality which implies that the $f_n$ are bijections.
\end{proof}

\begin{remark}
Since the functions $f_n$ are bijections then $\ga_s(\om)$ is in the domain of all $\psi_\nu$ for $\nu \in M^*$, for $\om = ([\mt]_n)$.
Hence $\ga_s$ fixes the points $([\mt]_n)$ of $\Om_{\La^*}$ and $\Om_{M^*}$.
\end{remark}

The entropy of a language $\La^*$ is given by $h(\La^*) := \lim_n n^{-1} \log_2 |\B_n(\La^*)|$.
The following corollary is immediate.

\begin{corollary}\label{C: entropy}
Local piecewise conjugate languages share the same entropy.
\end{corollary}

Local piecewise conjugacy respects the class of sofic languages.
However, as we will see in Example \ref{E: even vs sft}, it does not preserve the class of SFT's.

\begin{proposition}
Let $\La^*$ and $M^*$ be locally piecewise conjugate languages.
If $\La^*$ is  sofic then so is $M^*$.
\end{proposition}

\begin{proof}
Immediate by Proposition \ref{P: finite A}.
\end{proof}

\section{Applications to the follower set graph}

The quantized dynamics can be described by a finite graph in the case of the sofic languages.
Fix a sofic language $\La^*$ on $d$ symbols.
For every $\mu \in \La^*$ let
\[
F(\mu) : = \{w \in \La^* \mid w \mu \in \La^*\}
\]
be the follower set of $\mu$.
We write $F_{\La^*}(\mu)$ when we want to highlight the language to which we refer\footnote{\ We will rarely use this notation, as the language will be clear from the context.}.
Recall that the elements in $\Om$ are of the form $([\mu]_n)$ for $\mu \in \La^*$.
Since $([\mu]_n) = ([\nu]_n)$ if and only if $F(\mu) = F(\nu)$ we have a bijection between $\Om$ and the set $\{F(\mu) \mid \mu \in \La^*\}$.
By definition soficity is equivalent to $\Om$ being finite.
In this case Proposition \ref{P: finite A} provides the existence of a stabilizing step $k$ such that $\Om_l = \Om_k$ for all $l \geq k+1$; i.e. every $F(\mu)$ can be identified with $[\mu]_{k}$.

The \emph{follower set graph} $\fG_{\La^*} = (G_{\La^*}, \fL)$ of a sofic language on $d$ symbols is an edge-labeled graph, where $\fL$ is a colouring map from the edge set $G_{\La^*}^{(1)}$ onto $\{1, \dots, d\}$, defined as follows:
\begin{enumerate}
\item The vertices of $G_{\La^*}$ are given by the follower sets.
\item We draw an edge labeled $i$ from $F(\mu)$ to $F(i\mu)$ if and only if $F(i\mu) \neq \mt$.
\end{enumerate}
Therefore we have an edge (labeled $i$) between $\om = [\mu]_{k}$ and $\om' = [\nu]_{k}$ if $[i\mu]_{k} = [\nu]_{k}$, i.e. if $\vpi_i(\om) = \om'$.
It follows that the labeled graph gives a representation of the quantized dynamics.
Notice that this procedure gives also the follower set graph in the case of a two-sided subshift \cite[Section 3.2]{LM95} (when everything is written in the opposite direction).

Notice that we do not use soficity for the construction of the follower set graph, and one may be tempted to produce a follower set graph construction for any language.
However this is not the right thing to do.
As we saw in Section \ref{Ss: qd}, the spectrum of $A$ is totally disconnected and configuring the dynamics by a discrete structure would not comply with the topology.
By Proposition \ref{P: finite A}, it is exactly when $\La^*$ is sofic that the spectrum is discrete and we are able to picture the quantized dynamics by a graph without losing information about the topology. 

\begin{example}\label{E: sft 1}
Let $\La^*$ be of type $1$.
Then $\Om \simeq \Om_1 = \{[i_1]_1, \dots, [i_r]_1\}$ for some symbols $i_1, \dots, i_r \in \{1, \dots, d\}$.
For the follower set graph we have an edge labeled $j$ from $[i]_1 = F(i)$ to $[ji]_1 = [j]_1 = F(j)$ whenever $ji \in \La^*$.
Notice here the use of Proposition \ref{P: classes for sft}.

The case of type $1$ languages requires less information to store, as the label of an edge coincides with the one-lettered word that labels the vertex where the edge terminates.
\end{example}

\begin{remark}\label{R: aug}
There is a strong connection between the follower set graph of a sofic language $\La^*$ and that coming from its augmented subshift $(\wt{\La}, \si)$.
Suppose that $\La^*$ is on $d$ symbols and $\fG$ is its follower set graph.
Recall that $F(\mt) = F(\ze \mu)$ for every $\mu \in \wt{\La}^*$ for the added distinguished symbol $\ze$.
Therefore the follower set graph of $\wt{\La}^*$ contains $\fG$, shares the same vertex set, and contains in addition edges from every vertex to $F(\mt)$ labeled by $\ze$.
\end{remark}

When the language is of finite type then we can find a finite set of forbidden words that describes it.
By using this, we can determine the follower set graph in a finite number of steps.
Indeed the dual of Proposition \ref{P: classes for sft} suggests that the determination of every $F(\mu)$ can be achieved in finite steps, even though $F(\mu)$ may be infinite in principle.

\begin{proposition}\label{P: finite right}
Let $\La^*$ be a language of finite type $k$.
For every $\mt \neq \mu \in \La^*$ and $w_n \dots w_1 \in \La^*$ with $n \geq k$ we have that
\[
w \mu \in \La^* \qiff w_k \dots w_1 \mu \in \La^*.
\]
\end{proposition}

\begin{proof}
The proof is the same to that of Proposition \ref{P: classes for sft} once the words are reversed.
\end{proof}

Therefore, in order to distinguish between the $F(\mu)$ it suffices to distinguish between the finite sets
\[
\{w \in \B_k(\La^*) \mid w \mu \in \La^*\}.
\]
Consequently we obtain the following algorithm for constructing the follower set graph of a language of type $k$.

\begin{algo}\label{A: fsg sft}
Let $\La^*$ be a language of finite type and let $\F$ be a finite set of forbidden words that defines $\La^*$.
The algorithm takes $\F$ as input and has output the follower set graph.
We have two cases:

\medskip

\noindent $\qedsymbol$ {\bf Case 1.} If $F = \mt$ then $\La^* = \bF_+^d$ and its follower set graph is the Hawaiian ring on $d$ edges (lines 1--9 in the pseudocode).

\medskip

\noindent $\qedsymbol$ {\bf Case 2.} If $F$ is non-empty we write
\[
k := \max\{|\mu| \mid \mu \in \F\} - 1.
\]
Then $\La^*$ is a language of type $k$ (lines 10--16 in the pseudocode).
We then generate the set $\B_k(\La^*)$ of allowable words of length $k$ (lines 17--24 in the pseudocode).
We then proceed to forming the follower set graph in two steps:

\smallskip

$\bullet$ {\bf Step 1:} Determining the vertex set $\{F(\mu) \mid \mu \in \La^*\}$ (lines 25--44 in the pseudocode).
It suffices to compute the $F(\mu)$ for $|\mu| \leq k$ (Proposition \ref{P: classes for sft}).
Nevertheless we will also compute the $F(\mu)$ for $|\mu| = k+1$.
This will be helpful for writing the edges.
To this end we construct a table indexed by $\B_k(\La^*)$ whose entries show whether an allowable word can follow another one.
If two rows are the same then the follower sets of the corresponding words determine the same vertex on the graph.
It is convenient to multi-label a vertex by using all the $F(\mu)$, for $|\mu| \leq k+1$, that coincide.

- Form a table indexed by $\B_k(\La^*)$ (followerTable).
We fix an enumeration $\{\mu_1, \dots, \mu_N\}$ of these words, including the void word which we set to be $\mu_1$.
For the $(\mu_i, \mu_j)$ entry of the table write $\mu_j \mu_i$.
If $\mu_j \mu_i \in \La^*$ then write TRUE for the entry; otherwise write FALSE.
In this way we create a second table (truthTable).
By Proposition \ref{P: finite right} these finite entries are sufficient for identifying the follower sets.

- Compute the $F(\mu)$ for $\mu$ of length $k+1$ (lines 45--56 in the pseudocode).
If $\mu$ is forbidden then do nothing.
If $\mu = \mu_{1} \mu_k \dots \mu_{k+1}$ is allowable then add the label $F(\mu)$ to the vertex that has $F(\mu_1 \dots \mu_{k})$ among its labels.

\smallskip

$\bullet$ {\bf Step 2:} Determining the edges of the follower set graph (lines 57--69 in the pseudocode).
We read the rows of the truthTable from top to bottom.
If we pass to a row that is the same with one already read, we move to the next (i.e. we identify the words with the same follower sets.)

- Start from $\mu_l = \mu_1 = \mt$ and repeat for $l = 2, \dots, N$.
Let $\mu_l$ be a word from $\{\mu_1, \dots, \mu_N\}$.
Start with $i=1$ and repeat for $i=2, \dots, d$.
Write an edge labeled $i$ from the vertex with label $F(\mu_l)$ to the vertex with label $F(i \mu_l)$ if there is a vertex that contains the label $F(i \mu_l)$.
Otherwise do nothing and go to $i + 1$.
After we finish for $i = d$ (the last step), we repeat for $\mu_{l + 1}$.
\end{algo}

\begin{remark}
In the next pages we give the pseudocode for the algorithm.
Notice that in the output each node is a set of words, and each word associated with a node generates the same follower set.
We remind that we were not concerned about the technical features of this algorithm, but more about that it does terminate.

{\small
\begin{algorithm}
\caption{Generate FollowerSetGraph(F, symbolSet)}

\DontPrintSemicolon
\SetKwInOut{Input}{Input}
\SetKwInOut{Output}{Output}

\vspace{2pt}

\Input{
$F$ = the list of forbidden words generating the language.\\
$symbolSet = \{1, \dots, d\}$. 
}
\Output{$Graph = (Nodes, Edges)$ where: \linebreak
-- each node in $Nodes$ is the set of words of length less than or equal to $k+1$ which have the same follower set; and \linebreak
-- each edge in $Edges$ is of the form \emph{(source, destination, label)}. 
}
\BlankLine

\tcp{Deal with the case that $F$ is the empty set.}
	\If{$F$ \upshape{ is empty}}{
		initialize $node$ as a list containing $symbolSet$ and the empty word\;
		initialize $nodes$ as a list containing only $node$\;
		initialize $edges$ as empty list\;
		\ForEach{$letter \in symbolSet$}{
			insert $(node, node, letter)$ into $edges$
		}
		${\bf return} (nodes, edges)$\;	
	}

\tcp{If $F$ is non-empty find value of $k$ from $F$.}
initialize $maxLength$ as 1\;
	\ForEach{$word \in F$}{
		\If{\upshape{length of} $word > maxLength$}{
			$maxLength \longleftarrow $ length of $word$
		}
	}
initialize $k$ as $maxLength - 1$\;

\tcp{Generate $mu$, the set of allowed words of length at most $k$.}

initialize $mu$ as a list containing only the empty word\;
	\ForEach{$n \in \{1,2,\dots, k\}$}{
		\ForEach{$word \in$ \upshape{set of all possible words over} $symbolSet$ \upshape{of length} $n$}{
			\If{$word$ \upshape{contains no element of} $F$ \upshape{as a subword}}{
				insert $word$ into $mu$
			}
		}
	}

\tcp{Form tables. Indexed by base word and concatenated word.}
	\ForEach{{\upshape word} $a \in mu$}{
		\ForEach{{\upshape word} $b \in mu$}{
			$followerTable[a][b]$ $\longleftarrow$ concatenate $b$ with $a$ \;
			\eIf{$followerTable[a][b]$ \upshape{contains no element of} $F$ \upshape{as a subword}}{
				$truthTable[a][b] \longleftarrow True$ }
			{	
				$truthTable[a][b] \longleftarrow False$
			}
		}
	}
\BlankLine
\end{algorithm}
}

{\small
\begin{algorithm}
\setcounter{AlgoLine}{34}
\caption{Generate FollowerSetGraph(F, symbolSet), continued}

\vspace{8pt}

\tcp{Each node of the graph is a set of words sharing the same follower set.}
	initialize $Nodes$ as empty list\;
	\ForEach{$row \in$ {\upshape unique rows of} $truthTable$}{
		initialize $node$ as an empty list\;
		\ForEach{{\upshape word} $a \in mu$}{
			\If{$truthTable[a] = row$}{
				insert $a$ into $node$\;
			}
		}
		insert $node$ into $Nodes$\;
	}
\BlankLine	

\tcp{Add words of length $k+1$ to the appropriate node.}
	\ForEach{$node \in Nodes$}{
		\ForEach{$word \in node$}{
			\If{\upshape{length of} $word$ = $k$}{
				\ForEach{$letter \in symbolSet$}{
					$newWord \longleftarrow$ concatenate $word$ with $letter$\;
					\If{$newWord$ \upshape{contains no element of} $F$ \upshape{as a subword}}{
						insert $newWord$ into $node$\;
					}
				}
			}
		}
	}	
\BlankLine	

\tcp{Create edges.}
initialize $Edges$ as empty list\;
	\ForEach{$node \in Nodes$}{
		\ForEach{$letter \in symbolSet$}{
			$newWord \longleftarrow$ concatenate $letter$ with any word in $node$ of length less than $k+1$\;
			\tcp{Find which node $newWord$ belongs to, if any}
			\ForEach{$node' \in Nodes$}{
				\If{$newWord \in node'$}{
					$edge \longleftarrow (node,node',letter)$\;
					insert $edge$ into $Edges$\;
					${\bf break}$\;
				}
			}
		}
	}
\BlankLine
	
${\bf return} (Nodes, Edges)$\;	
\end{algorithm}
}
\end{remark}

Let us illustrate with an example how the follower set graph is constructed from the Follower Set Graph Algorithm.

\begin{example}\label{E: first}
Let the symbol set be $\{0,1\}$.
Let $\La^*$ be the language defined by the words $\F = \{101, 110\}$.
We form the followerTable of the Follower Set Graph Algorithm with all the words of length at most two.
As we add words on the left, it is convenient to put the labels at the right of the followerTable:

\smallskip

\begin{center}
\begin{tabular}{|c|c|c|c|c|c|c||l|}
\hline
11 & 01 & 10 & 00 & 1 & 0 & $\mt$ & \\
\hline
\hline
11 & 01 & 10 & 00 & 1 & 0 & $\mt$ & $\mt$ \\
\hline
110 & 010 & 100 & 000 & 10 & 00 & 0  & 0 \\
\hline
111 & 011 & 101 & 001 & 11 & 01 & 1 & 1 \\
\hline
1100 & 0100 & 1000 & 0000 & 100 & 000 & 00 & 00 \\
\hline 
1110 & 0110 & 1010 & 0010 & 110 & 010 & 10 & 10 \\
\hline 
1101 & 0101 & 1001 & 0001 & 101 & 001 & 01 & 01 \\
\hline 
1111 & 0111 & 1011 & 0011 & 111 & 011 & 11 & 11 \\
\hline 
\end{tabular}
\end{center}

\smallskip

\noindent For simplicity we can just indicate the forbidden words by a black square in the truthTable of the Follower Set Graph Algorithm:

\smallskip

\begin{center}
\begin{tabular}{|c|c|c|c|c|c|c||l|}
\hline
11 & 01 & 10 & 00 & 1 & 0 & $\mt$ & \\
\hline
\hline
 &  &  &  &  &  &  & $\mt$ \\
\hline
$\qedsymbol$ &  &  &  &  &  &   & 0 \\
\hline
 & & $\qedsymbol$ &  &  &  &  & 1 \\
\hline
$\qedsymbol$ & & &  & &  &  & 00 \\
\hline 
$\qedsymbol$ & $\qedsymbol$ & $\qedsymbol$ &  & $\qedsymbol$ &  &  & 10 \\
\hline 
$\qedsymbol$ & $\qedsymbol$ &  &  & $\qedsymbol$ &  &  & 01 \\
\hline 
 &  & $\qedsymbol$ &  &  &  &  & 11 \\
\hline 
\end{tabular}

%
\end{center}

\smallskip

\noindent This gives the following five vertices:
\begin{align*}
& F(\mt), \\
& F(0) = F(00) = F(000) = F(001), \\
& F(1) = F(11) = F(111), \\
& F(10) = F(100), \\
& F(01) = F(010) = F(011),
\end{align*}
where we used Proposition \ref{P: classes for sft} to provide the classes for the allowable words of length $3$.
Notice here that $F(101)$ and $F(110)$ are meaningless as $101$ and $110$ are forbidden words.
Then the follower set graph is:
{\small
\[
\xymatrix@R=10mm@C=3mm{
& {F(\mt)} \ar@{-->}^{1}@/^1pc/[ddr] \ar@{->}_{0}@/_1pc/[ddl] & \\
& & \\
{F(0) = F(00) = F(000) = F(001)} \ar @`{(-35,5),(5,5)}^{0}  \ar@{-->}_{1}@/_1pc/[dd] & & {F(1) = F(11) = F(111)} \ar@{-->} @`{(50,5),(100,5)}^{1} \ar@{->}^{0}@/^1pc/[dd] \\
& & \\
{F(10) = F(100)} \ar@{->}_{0}@/_2pc/[rr] & & {F(01) = F(010) = F(011)} \ar@{->}_{0}@/_1pc/[uull] & \\
}
\]
}
\end{example}

By definition, the follower set graph is \emph{(left-)resolving}, i.e. different edges with the same source carry different labels.
Recall that a graph is called \emph{follower-separated} if distinct vertices have distinct follower sets (set of paths starting on the vertex).
This is also a property that the follower set graph has.
Indeed, if two vertices, say $F(\mu)$ and $F(\nu)$ have the same follower sets on the graph then every path $w$ that starts from $\mu$ can also start at $\nu$.
This shows that $w\mu \in \La^*$ if and only if $w \nu \in \La^*$ for all $w \in \La^*$, hence $F(\mu) = F(\nu)$.

Resolving graphs go by the name of \emph{Shannon graphs} in the literature.
Another type of a Shannon graph for subshifts is produced through the \emph{Krieger cover} \cite{Kri84}.
The Krieger cover is constructed by taking the follower sets on one-way infinite words.
This is the important difference with the follower set graph here, as we consider the vertices to be the follower sets on the \emph{finite} words.
For a nice exposition (among others) on the Krieger cover, the reader can see also \cite{Joh12}.

Given two languages we can consider their labeled graphs or the ambient unlabeled graphs.
Graph isomorphism within each class is translated to a different level of equivalence.

\subsection{Labeled graph isomorphism}

Let $\fG_1 = (G_1, \fL_1)$ and $\fG_2 = (G_2, \fL_2)$ be two edge-labeled graphs such that the labels $\fL_1, \fL_2$ take values on the same symbol set $\{1, \dots, d\}$.
The labeled graphs $\fG_1$ and $\fG_2$ are called \emph{labeled graph isomorphic} if there is a graph isomorphism $(\partial \ga, \ga) \colon G_1 \to G_2$ and a bijection $\pi$ on the symbol sets such that \[
\fL_2(\ga(e)) = \pi(\fL_1(e)) \text{ for all edges $e$}.
\]
(We use the notation $\partial\ga \colon G_1^{(0)} \to G_2^{(0)}$ and $\ga \colon G_1^{(1)} \to G_2^{(1)}$ for the graph isomorphism.)

\begin{proposition}\label{P: 1 block code}
Let $\La^*$ and $M^*$ be sofic languages.
Then their follower set graphs are isomorphic if and only if there is a bijection on their symbol sets that preserves the allowable words.
\end{proposition}

\begin{proof}
Suppose there is an edge-labeled graph isomorphism.
By construction there is at least one vertex, i.e. the vertex $F_{\La^*}(\mt)$ (resp. $F_{M^*}(\mt)$), that emits edges of all labels.
Hence $\max\{s^{-1}(v) | v \text{ vertex}\}$ coincides with the number of the symbols, thus $d = d'$.
The graph isomorphism identifies $F_{\La^*}(\mt)$ with some $F_{M^*}(\nu)$.
Then the edge $e = (F_{\La^*}(\mt), F_{\La^*}(i))$ labeled $i$ corresponds to the edge
\[
\ga(e) = (F_{M^*}(\nu), F_{M^*}(\pi(i)\nu))
\] labeled $\pi(i)$.
Suppose $ji \in \La^*$.
Then the edge $f = (F_{\La^*}(i), F_{\La^*}(ji))$ labeled $j$ corresponds to the edge 
\[
\ga(f) = (F_{M^*}(\pi(i)\nu), F_{M^*}(\pi(j)\pi(i)\nu))
\]
labeled $\pi(j)$.
Therefore we get that $\pi(j) \pi(i) \in M^*$.
Hence $\pi$ respects the allowable words of length $2$.
Inductively we get that permutation of the edges along the vertices extends to a match on the allowable words of any length.
Notice here that $\pi^{-1}$ is the associated labeling for the inverse graph isomorphism.

Conversely, the existence of a bijection $\pi$ on the symbol sets implies that $d = d'$.
Furthermore $\pi$ extends to the allowable words such that $\pi(\mu \nu) = \pi(\mu) \pi(\nu)$.
We define the edge-labeled graph homomorphism by
\[
\partial\ga(F_{\La^*}(\mu)):= F_{M^*}(\pi(\mu)).
\]
The map $\partial\ga$ is well defined.
Indeed if $F_{\La^*}(\mu) = F_{\La^*}(\nu)$ then we get that
\[
\{w \in \La^* \mid w \mu \in \La^*\} = \{ w \in \La^* \mid w \nu \in \La^*\}.
\]
Applying $\pi$ and working backwards we have that $F_{M^*}(\pi(\mu)) = F_{M^*}(\pi(\nu))$.
It follows that $\partial\ga$ is also one-to-one, hence a bijection due to the properties of $\pi$.
If there is an edge $e = (F_{\La^*}(\mu), F_{\La^*}(i\mu))$ labeled $i$ then we have that $i\mu \in \La^*$ and thus $\pi(i \mu) = \pi(i) \pi(\mu) \in M^*$.
Therefore there is an edge labeled $\pi(i)$ connecting $F_{M^*}(\pi(\mu))$ and $F_{M^*}(\pi(i\mu))$; we write $\ga(e)$ for that edge.
Since the edges on the same source have different labels we have that $\ga(e)$ is unique.
In this way, we extend $\partial\ga$ to a map $\ga$ on the edges.
Then $(\partial\ga, \ga)$ gives the required edge-labeled graph isomorphism.
\end{proof}

\begin{remark}
Proposition \ref{P: 1 block code} is straightforward when $\La^*$ and $M^*$ are languages of subshifts.
In this case $(X_{\La^*})^* = \La^*$ and the bijection gives a $1$-block code between the subshifts.
However here we provide the same result even when $(X_{\La^*})^* \subset \La^*$ and with no reference to possible subshifts that the languages induce.
\end{remark}

\subsection{Unlabeled graph isomorphisms}

By an unlabeled graph isomorphism between two edge-labeled graphs we will mean a simple isomorphism of the ambient graphs with the labels omitted.

\begin{theorem}\label{T:unlbl iso}
Let $\La^*$ and $M^*$ be sofic languages.
Then their unlabeled follower set graphs are isomorphic if and only if their quantized dynamics are locally piecewise conjugate.
\end{theorem}

\begin{proof}
By assumption the spaces $\Om_{\La^*}$ and $\Om_{M^*}$ are (discrete) finite spaces.
First suppose that the unlabeled follower set graphs are isomorphic.
Then we obtain an induced bijection $\ga_s$ on the vertices, and thus a homeomorphism $\ga_s \colon \Om_{\La^*} \to \Om_{M^*}$.
Let a point $\om \in Q_{[m]}$.
Then the number of the edges emitted by $\om$ coincides with the number of edges emitted by $\ga_s(\om)$.
Due to the follower set graph construction, this implies that $\ga_s(\om) \in P_{[n]}$, with $|\supp [n]|$ coinciding with the number of emitted edges from $\ga_s(\om)$, thus with $|\supp [m]|$.
Now the graph isomorphism implies a bijection, say $\pi$, between edges.
For convenience let $\supp [m] = \{i_1, \dots, i_r\}$ and $\supp [n] = \{j_1, \dots, j_r\}$, such that $\pi(i_l) = j_l$ for all $l=1, \dots r$.
Then the terminal vertex of $i_l$ is mapped to the terminal of $\pi(i_l)$.
Consequently we obtain
\[
\ga_s \phi_{i_l}(\om) = \psi_{\pi(i_l)} \ga_s(\om) \foral l = 1, \dots, r.
\]
Taking $\U = \{\om\}$ gives the required local piecewise conjugacy.

Conversely suppose that $(\Om_{\La^*}, \vpi)$ and $(\Om_{M^*}, \psi)$ are locally piecewise conjugate.
Then the ambient spaces are homeomorphic, i.e. there is a bijection on the vertex sets.
Moreover for every point $\om \in Q_{[m]}$ we have $\ga_s(\om) \in P_{[n]}$ with $|\supp [m]| = |\supp [n]|$, such that
\begin{align*}
\ga_s \vpi_i(\om) = \psi_{\pi(i)} \ga_s(\om) \foral i \in \supp [m].
\end{align*}
Recall that $\vpi_i$ is defined on $\om$ if and only if $i \in \supp [m]$.
Hence local piecewise conjugacy gives that the number of edges that $\om$ emits is $|\supp [m]|$, and thus it equals the number of edges that $\ga_s(\om)$ emits, which is $|\supp [n]|$.
The above equation then shows that for the edge labeled $i$ there is a unique edge labeled $\pi(i)$, such that the end point of the $i$-edge is mapped to the endpoint of the $\pi(i)$-edge.
As edges are preserved under $\ga_s$ we get the required graph isomorphism.
\end{proof}

\begin{remark}
There is a conceptual difference between labeled and unlabeled graph isomorphism.
In the first case the isomorphism on the edges is given by the bijection on the labels, and is the same on all vertices.
However in the second case the bijection changes each time we pass to another vertex.
\end{remark}

Removing the labels from a graph representation $\fG = (G, \fL)$ of a subshift $\La$ does not preserve type.
In the case of sofic subshifts this procedure amounts to producing finite covers and hints that unlabeled graph isomorphisms of the follower set graphs \emph{should not}\footnote{\ 
Notice the subtle point here: 
If $\La$ is a sofic subshift then removing all labels from its follower set graph $\fG = (G, \fL)$ produces an edge-shift $X_G$, and $\La$ is a factor of $X_G$.
However it is not ensured that $G$ is the follower set graph of $X_G$.
} preserve SFT's.
The following example clarifies this point.

\begin{example}\label{E: even vs sft}
Let $\La^*$ be the language of the even shift on $\{0,1\}$, i.e. the forbidden words are of the form $10^{2n+1}1$ for all $n\geq 0$.
Then the vertices of the follower set graph are of the form
\[
F(\mu)
=
\begin{cases}
F(\mt) & \text{ if $\mu$ contains no $1$'s},\\
F(1) & \text{ if $\mu$ begins with $0^{2k}1$ for some $k \geq 0$}, \\
F(01) & \text{ if $\mu$ begins with $0^{2k + 1}1$ for some $k \geq 0$}.
\end{cases}
\]
for $\La^*$.
Therefore the follower set graph takes up the form
{\small
\[
\xymatrix@R=30pt{
& & F(\mt) \ar@{->}^{0}@(r,dr) \ar@{-->}^{1}@/_1pc/[ddll] & & \\
& & & & \\
F(1) \ar@{-->}_{1}@(ul,l) \ar@{->}^{0}@/^1pc/[rrrr] & & & & F(01) \ar@{->}^{0}@/^1pc/[llll]
}
\]
}For more details see \cite[Example 3.2.7, Figure 3.2.2]{LM95}.

On the other hand let $M^*$ on $\{0,1\}$ be defined by the forbidden word $001$.
Then the truthTable of the Follower Set Graph Algorithm is given by:

\smallskip

\begin{center}
\begin{tabular}{|c|c|c|c|c|c|c||l|}
\hline
11 & 01 & 10 & 00 & 1 & 0 & $\mt$ & \\
\hline
\hline
& & & & & & & $\mt$ \\
\hline
& & & & & & & 0 \\
\hline
& & & $\qedsymbol$ & & & & 1 \\
\hline
& & & & & & & 00 \\
\hline 
& & & $\qedsymbol$ & & & & 10 \\
\hline 
& & $\qedsymbol$ & $\qedsymbol$ & & $\qedsymbol$ & & 01 \\
\hline 
& & & $\qedsymbol$ & & & & 11 \\
\hline 
\end{tabular}
\end{center}

\smallskip

\noindent (with FALSE indicated by a black box).
Therefore we have the vertices
\begin{align*}
F(\mt) & = F(0) = F(00) = F(000), \\
F(1) & = F(10) = F(11) = F(100) = F(101) = F(110) = F(111), \\
F(01) & = F(010) = F(011) .
\end{align*}
Then the follower set graph for $M^*$ is
{\small
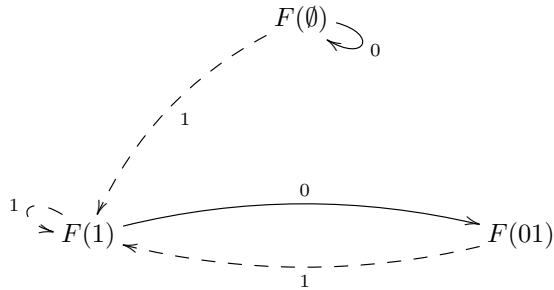
\[
\xymatrix@R=30pt{
& & F(\mt) \ar@{->}^{0}@(r,dr) \ar@{-->}^{1}@/_1pc/[ddll] & & \\
& & & & \\
F(1) \ar@{-->}_{1}@(ul,l) \ar@{->}^{0}@/^1pc/[rrrr] & & & & F(01) \ar@{-->}^{1}@/^1pc/[llll]
}
\]
}
It is immediate that the unlabeled graphs for $\La^*$ and $M^*$ are isomorphic.
However $\La^*$ is not a language of finite type.

Notice also how the functions $f_n \colon \B_n(\La^*) \to \B_n(M^*)$ of Remark \ref{R: fn} depend on $n$ in this example.
For example, we have that
\begin{align*}
f_4(0010) = 1010 \neq 0 010 = f_1(0) f_3(010) .
\end{align*}
Recall that the image of $w$ under $f_n$ is taken through the unlabeled graph isomorphism, when $w$ corresponds to \emph{the} path $w$ beginning at $\om = F(\mt)$.
\end{example}

\begin{remark}\label{R:zeta fun 2sided}
Recall that a point $x$ in a two-sided subshift $\La$ has \emph{period $n$} if there exists an $n \in \bN$ such that $\si^n(x) = x$.
The number of points with period $n$ is denoted by $p_n(\La)$.
The \emph{zeta function} of $\La$ is given by
\[
\zeta_{\La}(t) := \exp ( \, \sum_{n=1}^\infty \frac{p_n(\La)}{n} t^n \,).
\]

The zeta function is not preserved by local piecewise conjugacy.
In Example \ref{E: even vs sft} we show that the even shift is locally piecewise conjugate to a subshift of finite type.
By \cite[Theorem 6.4.6]{LM95} the zeta function of any subshift of finite type is the reciprocal of a polynomial, whereas \cite[Example 6.4.5]{LM95} implies that the zeta function of the even shift is $\zeta(t) = (1+t)(1-t-t^2)^{-1}$.
We would like to thank Ian Putnam for this remark.
\end{remark}

\subsection{Graph isomorphism for type $1$ languages}

Let us examine further the case of type $1$ languages.
We begin with an example of two languages with no isomorphic labeled follower set graphs.

\begin{example}\label{E:not same sft}
Consider the language $\La^*$ on $5$ symbols determined by the forbidden words
\[
\{11, 21, 31, 41, 12, 22, 32, 42, 13, 33, 24, 44\}.
\]
Then the resulting follower set graph is
{\small
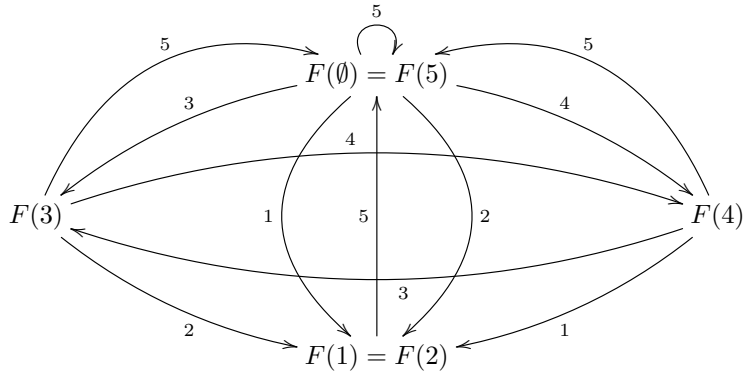
\[
\xymatrix@R=5.5mm@C=30mm{
& F(\mt) = F(5) \ar@{->}_{1}@/_3pc/[dddd] \ar@{->}^{2}@/^3pc/[dddd] \ar@{->}^{4}@/^1pc/[ddr]  \ar@{->}_{3}@/_1pc/[ddl] \ar@{->}^{5}@(ul,ur) & \\
& & \\
F(3) \ar@{->}_{2}@/_1pc/[ddr] \ar@{->}^{4 \qquad}@/^2pc/[rr] \ar@{->}^{5}@/^3pc/[uur] & & F(4) \ar@{->}^{1}@/^1pc/[ddl] \ar@{->}^{\qquad 3}@/^2pc/[ll] \ar@{->}_{5}@/_3pc/[uul]\\
& & \\
& F(1) = F(2) \ar@{->}^{5}@/_0pc/[uuuu] & \\
}
\]
}for $\La^*$.
Consider also the language $M^*$ on $5$ symbols determined by the forbidden words
\[
\{11, 21, 31, 41, 12, 22, 32, 42, 13, 33, 14, 44\}.
\]
The only difference with the forbidden words of $\La^*$ is to consider $14$ in place of $24$.
Similarly we get the follower set graph
{\small
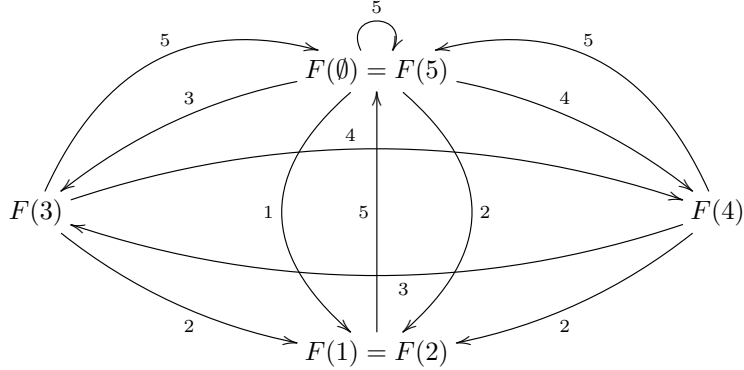
\[
\xymatrix@R=5.5mm@C=30mm{
& F(\mt) = F(5) \ar@{->}_{1}@/_3pc/[dddd] \ar@{->}^{2}@/^3pc/[dddd] \ar@{->}^{4}@/^1pc/[ddr]  \ar@{->}_{3}@/_1pc/[ddl] \ar@{->}^{5}@(ul,ur) & \\
& & \\
F(3) \ar@{->}_{2}@/_1pc/[ddr] \ar@{->}^{4 \qquad}@/^2pc/[rr] \ar@{->}^{5}@/^3pc/[uur] & & F(4) \ar@{->}^{2}@/^1pc/[ddl] \ar@{->}^{\qquad 3}@/^2pc/[ll] \ar@{->}_{5}@/_3pc/[uul]\\
& & \\
& F(1) = F(2) \ar@{->}^{5}@/_0pc/[uuuu] & \\
}
\]
}for $M^*$.
The unlabeled graphs are isomorphic.
The only difference in the labeled graphs is the lower right arrow which carries different labels.

As the follower set graphs are irreducible it can be seen that the languages $\La^*$ and $M^*$ are also the languages of two-sided irreducible subshifts.
In fact these are the augmented versions of \cite[Example 9.8]{KS15}.
\end{example}

A key feature in the example above is that there are two symbols that have the same follower sets.
It appears that this is the only obstruction.

\begin{theorem}\label{T: same graph iso}
Let $\La^*$ and $M^*$ be languages of type $1$.
Suppose that $\La^*$ is on $d$ symbols and that $F_{\La^*}(i) \neq F_{\La^*}(j)$ for $i \neq j$, with $i,j = 1, \dots d$.
The following are equivalent:
\begin{enumerate}
\item The follower set graphs are isomorphic;
\item The unlabeled follower set graphs are isomorphic.
\end{enumerate}
\end{theorem}

\begin{proof}
Of course item (i) implies item (ii).
For the converse recall that unlabeled graph isomorphism imposes that $M^*$ is on $d$ symbols as well.
Furthermore the vertex sets must have the same size.

We have two cases.
If $F_{\La^*}(i) \neq F_{\La^*}(\mt)$ for all $i$ then the same must hold for the graph of $M^*$, as $d+1$ is the maximum size of the vertex sets.
In this case both $F_{\La^*}(\mt)$ and $F_{M^*}(\mt)$ are the unique sources for the graph, hence related by the graph isomorphism.
If there is an $i$ such that $F_{\La^*}(i) = F_{\La^*}(\mt)$ then the number of the $F_{M^*}(j)$ is $d$, and thus there exists a unique $j$ such that $F_{M^*}(j) = F_{M^*}(\mt)$.

In any case we get that the graph isomorphism induces a bijection between $\{F_{\La^*}(i) \mid i = 1, \dots, d\}$ and $\{ F_{M^*}(i) \mid i= 1, \dots, d \}$.
Without loss of generality we may relabel for $M^*$ so that this bijection sends $F_{\La^*}(i)$ to $F_{M^*}(i)$.
As we remarked in Example \ref{E: sft 1} the labels on the edges are pre-determined by their range.
Notice that by hypothesis every vertex receives at most one edge.
Therefore the unlabeled graph isomorphism respects the label of the edges, and the proof is complete.
\end{proof}

\begin{remark}
For type $1$ subshifts there is a strong connection between the followerTable of the Follower Set Graph Algorithm and the representation of the subshift as an edge shift.
Let us recall how this follows from \cite[Theorem 2.3.2]{LM95}.
To allow comparisons we denote the graph of the edge shift by $G_e(\La)$.

Let $\La$ be a two-sided subshift of type $1$.
The vertices of $G_e(\La)$ are the symbols of $\La$.
We write an edge between the vertex $i$ and $j$ if (and only if) $ji \in \La^*$, and label the resulting edge by $j$.
It is evident that this graph is given by a $0$-$1$ adjacency matrix.

If $A$ is the followerTable of the Follower Set Graph Algorithm where we replace the allowable words by $1$ and the forbidden words by $0$, then the adjacency matrix of $G_e(\La)$ is taken by deleting the row and the column that corresponds to $\mt$ in $A$ (which all have entries equal to $1$).

Working under the condition that $F$ is one-to-one on the symbol set (of Theorem \ref{T: same graph iso}) we distinguish two cases:

$\bullet$ Case 1. If $F(\mt) = F(i)$ for some symbol $i$ then $G_e(\La)$ coincides with the graph $G$ of the follower set graph $\fG = (G, \fL)$ of $\La$.

$\bullet$ Case 2. If $F(\mt) \neq F(i)$ for all symbols $i$ then $G_e(\La)$ coincides with the subgraph $G$ of the follower set graph $\fG = (G, \fL)$ of $\La$, once we erase the vertex $F(\mt)$ and the emitting edges.
\end{remark}

\begin{remark}
Theorem \ref{T: same graph iso} applies to edge shifts with invertible adjacency matrices.
Indeed let $A_e$ be the adjacency matrix of $G_e(\La)$.
If there are $i \neq j$ with $F(i) = F(j)$ then we have that two rows of the matrix from the Follower Set Graph Algorithm coincide.
Thus the same holds for $A_e$ and thus $\det A_e =0$.
\end{remark}

\begin{remark}\label{R: sft 1 symb 2}
Theorem \ref{T: same graph iso} holds in the particular case of languages of type $1$ on two symbols $\{0, 1\}$ without the assumption on the follower set function.
Indeed the cases where $F(0) = F(1)$ produce the following graphs
{\small
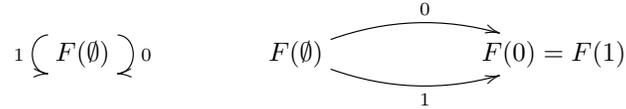
\[
\xymatrix{
F(\mt) \ar@{->}^{0}@(ur,dr) \ar@{->}_{1}@(ul,dl) & & F(\mt) \ar@{->}^{0}@/^1pc/[rr] \ar@{->}_{1}@/_1pc/[rr] & & F(0) = F(1)
}
\]
}
for the sets of forbidden words $\mt$ and $\{00, 10, 01, 11\}$, and 
{\small
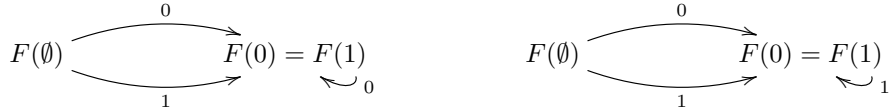
\[
\xymatrix{
F(\mt) \ar@{->}^{0}@/^1pc/[rr] \ar@{->}_{1}@/_1pc/[rr] & & F(0) = F(1) \ar@{->}^{0}@(r,dr) & & F(\mt) \ar@{->}^{0}@/^1pc/[rr] \ar@{->}_{1}@/_1pc/[rr] & & F(0) = F(1) \ar@{->}^{1}@(r,dr)
}
\]
}for the sets of forbidden words $\{00, 01\}$ and $\{11, 10\}$, respectively.
It follows that also in these cases the follower set graphs are unique up to a permutation of symbols.
\end{remark}

However one direction of Theorem \ref{T: same graph iso} does not hold in general even for languages of finite type.
We highlight this in the following two examples.

\begin{example}\label{E:counter1}
Let the language $\La^*$ on two symbols $\{0,1\}$ be defined by the forbidden words
\[
000, 010, 001, 101, 011.
\]
Then the corresponding truthTable of the Follower Set Graph Algorithm is:

\smallskip

\begin{center}
\begin{tabular}{|c|c|c|c|c|c|c||l|}
\hline
11 & 01 & 10 & 00 & 1 & 0 & $\mt$ & \\
\hline
\hline
 &  &  &  &  &  &  & $\mt$ \\
\hline
 & $\qedsymbol$ &  & $\qedsymbol$ &  &  &  & 0 \\
\hline
 & $\qedsymbol$ & $\qedsymbol$ & $\qedsymbol$ &  &  &  & 1 \\
\hline
 & $\qedsymbol$ & $\qedsymbol$ & $\qedsymbol$ &  & $\qedsymbol$ &  & 00 \\
\hline 
 & $\qedsymbol$ & $\qedsymbol$ & $\qedsymbol$ & & $\qedsymbol$ & & 10 \\
\hline 
$\qedsymbol$ & $\qedsymbol$ & $\qedsymbol$ & $\qedsymbol$ & $\qedsymbol$ & $\qedsymbol$ &  & 01 \\
\hline 
 & $\qedsymbol$ & $\qedsymbol$ & $\qedsymbol$ & & $\qedsymbol$ & & 11 \\
\hline 
\end{tabular}
\end{center}

\smallskip

\noindent and gives the vertices
\begin{align*}
& v_0 = F(\mt), v_1 = F(1), v_2 = F(0), v_3 = F(01), \\
& v_4 = F(00) = F(10) = F(11) = F(100) = F(110) = F(111).
\end{align*}
Then the follower set graph of $\La^*$ is 

\smallskip

{\small
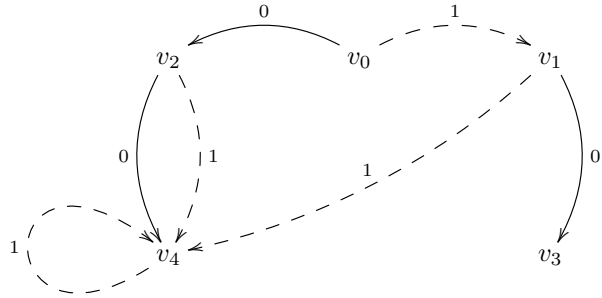
\[
\xymatrix@R=10mm@C=20mm{
v_2  \ar@{->}_{0}@/_1pc/[dd] \ar@{-->}^{1}@/^1pc/[dd] & v_0 \ar@{-->}^{1}@/^1pc/[r] \ar@{->}_{0}@/_1pc/[l] & v_1 \ar@{->}^{0}@/^1pc/[dd] \ar@{-->}_{1}@/^1pc/[ddll] \\
& & \\
v_4 \ar@{-->} @`{(-25,-45),(-25,-5)}^{1} & & v_3
}
\]
}

\bigskip

\noindent On the other hand let the language $M^*$ on two symbols $\{0,1\}$ be defined by the forbidden words
\[
000, 010, 001, 100, 011.
\]
Then the corresponding truthTable of the Follower Set Graph Algorithm is:

\smallskip

\begin{center}
\begin{tabular}{|c|c|c|c|c|c|c||l|}
\hline
11 & 01 & 10 & 00 & 1 & 0 & $\mt$ & \\
\hline
\hline
 &  &  &  &  &  & & $\mt$ \\
\hline
 & $\qedsymbol$ & $\qedsymbol$ & $\qedsymbol$ &  &  &  & 0 \\
\hline
 & $\qedsymbol$ &  & $\qedsymbol$ &  &  &  & 1 \\
\hline
$\qedsymbol$ & $\qedsymbol$ & $\qedsymbol$ & $\qedsymbol$ & $\qedsymbol$ & $\qedsymbol$ &  & 00 \\
\hline 
 & $\qedsymbol$ & $\qedsymbol$ & $\qedsymbol$ & & $\qedsymbol$ & & 10 \\
\hline 
 & $\qedsymbol$ & $\qedsymbol$ & $\qedsymbol$ & & $\qedsymbol$ & & 01 \\
\hline 
 & $\qedsymbol$ & $\qedsymbol$ & $\qedsymbol$ & & $\qedsymbol$ & & 11 \\
\hline 
\end{tabular}
\end{center}

\smallskip

\noindent and gives the vertices
\begin{align*}
& w_0 = F(\mt), w_1 = F(0), w_2 = F(1), w_3 = F(00), \\
& w_4 = F(10) = F(01) = F(11) = F(101) = F(111) = F(110).
\end{align*}
Then the follower set graph of $M^*$ is 

\smallskip

{\small
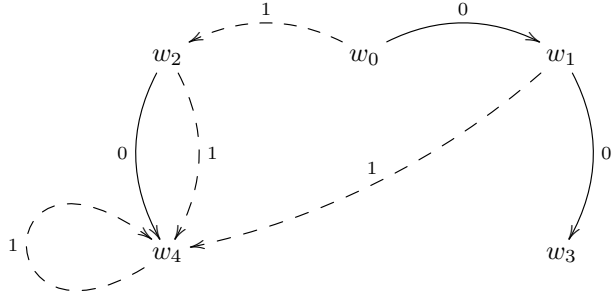
\[
\xymatrix@R=10mm@C=20mm{
w_2  \ar@{->}_{0}@/_1pc/[dd] \ar@{-->}^{1}@/^1pc/[dd] & w_0 \ar@{->}^{0}@/^1pc/[r] \ar@{-->}_{1}@/_1pc/[l] & w_1 \ar@{->}^{0}@/^1pc/[dd] \ar@{-->}_{1}@/^1pc/[ddll] \\
& & \\
w_4 \ar@{-->} @`{(-25,-45),(-25,-5)}^{1} & & w_3
}
\]
}

\bigskip

It is clear that there is only one unlabeled graph isomorphism; the one sending $v_i$ to $w_i$.
If it lifted to a labeled graph isomorphism then the $0$ label would match to the $0$ label, as it appears from $v_4$ and $w_4$.
However this does not comply with the labels on $v_0$ and $w_0$.

The two-sided subshifts coming from the languages $\La^*$ and $M^*$ are formed on a single point.
For creating a more interesting counterexample in the category of two-sided subshifts we may use the augmentations defined in $\{0,1,\ze\}$.
In this case we obtain

\smallskip

{\small
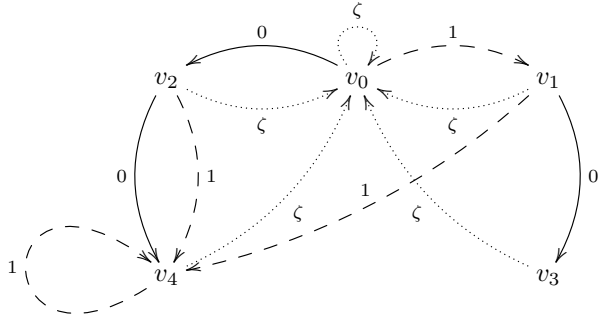
\[
\xymatrix@R=10mm@C=20mm{
v_2 \ar@{.>}_{\ze}@/_1pc/[r]  \ar@{->}_{0}@/_1pc/[dd] \ar@{-->}^{1}@/^1pc/[dd] & v_0 \ar@{.>}^{\ze}@(ul,ur) \ar@{-->}^{1}@/^1pc/[r] \ar@{->}_{0}@/_1pc/[l] & v_1 \ar@{.>}^{\ze}@/^1pc/[l] \ar@{->}^{0}@/^1pc/[dd] \ar@{-->}_{1}@/^1pc/[ddll] \\
& & \\
v_4 \ar@{.>}_{\ze}@/_1pc/[uur] \ar@{-->} @`{(-25,-45),(-25,-5)}^{1} & & v_3 \ar@{.>}^{\ze}@/^1pc/[uul]
}
\]
}

\bigskip

\noindent as the follower set graph of $\wt{\La}^*$ and

\smallskip

{\small
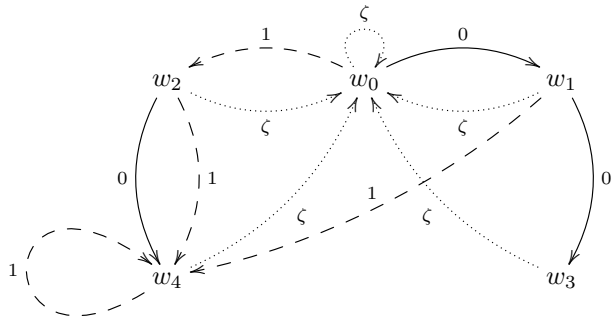
\[
\xymatrix@R=10mm@C=20mm{
w_2 \ar@{.>}_{\ze}@/_1pc/[r] \ar@{->}_{0}@/_1pc/[dd] \ar@{-->}^{1}@/^1pc/[dd] & w_0 \ar@{.>}^{\ze}@(ul,ur) \ar@{->}^{0}@/^1pc/[r] \ar@{-->}_{1}@/_1pc/[l] & w_1 \ar@{.>}^{\ze}@/^1pc/[l] \ar@{->}^{0}@/^1pc/[dd] \ar@{-->}_{1}@/^1pc/[ddll] \\
& & \\
w_4 \ar@{.>}_{\ze}@/_1pc/[uur] \ar@{-->} @`{(-25,-45),(-25,-5)}^{1} & & w_3 \ar@{.>}^{\ze}@/^1pc/[uul]
}
\]
}

\bigskip

\noindent as the follower set graph of $\wt{M}^*$.
Again there is not a labeled graph isomorphism between those.
\end{example}

\begin{example}\label{E:counter2}
Let the language $\La^*$ on two symbols $\{0,1\}$ be defined by the forbidden words
\[
000, 100, 010, 101, 011, 111.
\]
Then the corresponding truthTable of the Follower Set Graph Algorithm is:

\smallskip

\begin{center}
\begin{tabular}{|c|c|c|c|c|c|c||l|}
\hline
11 & 01 & 10 & 00 & 1 & 0 & $\mt$ & \\
\hline
\hline
 &  &  &  &  &  & & $\mt$ \\
\hline
 & $\qedsymbol$ & $\qedsymbol$ & $\qedsymbol$ &  &  &   & 0 \\
\hline
$\qedsymbol$ & $\qedsymbol$ & $\qedsymbol$ &  &  &  &  & 1 \\
\hline
$\qedsymbol$ & $\qedsymbol$ & $\qedsymbol$ & $\qedsymbol$ & $\qedsymbol$ & $\qedsymbol$ &  & 00 \\
\hline 
$\qedsymbol$ & $\qedsymbol$ & $\qedsymbol$ & $\qedsymbol$ &  & $\qedsymbol$ &  & 10 \\
\hline 
$\qedsymbol$ & $\qedsymbol$ & $\qedsymbol$ & $\qedsymbol$ & $\qedsymbol$ &  &  & 01 \\
\hline 
$\qedsymbol$ & $\qedsymbol$ & $\qedsymbol$ & $\qedsymbol$ & $\qedsymbol$ & $\qedsymbol$ &  & 11 \\
\hline 
\end{tabular}
\end{center}

\smallskip

\noindent and gives the vertices
\begin{align*}
& v_0 = F(\mt), v_1 = F(0), v_2 = F(1), v_3 = F(10), v_4 = F(01) \\
& v_5 = F(00) = F(11) = F(110) = F(001).
\end{align*}
Then the follower set graph of $\La^*$ is
{\small
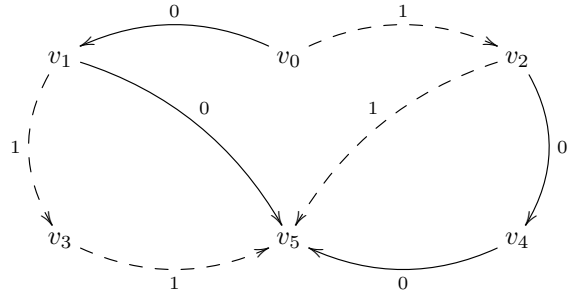
\[
\xymatrix@R=20mm@C=25mm{
v_1  \ar@{-->}_{1}@/_1pc/[d] \ar@{->}^{0}@/^1pc/[dr] & v_0 \ar@{-->}^{1}@/^1pc/[r] \ar@{->}_{0}@/_1pc/[l] & v_2 \ar@{->}^{0}@/^1pc/[d] \ar@{-->}_{1}@/_1pc/[dl]\\
v_3  \ar@{-->}_{1}@/_1pc/[r] & v_5 & v_4 \ar@{->}^{0}@/^1pc/[l]
}
\]
}

\smallskip

\noindent On the other hand let the language $M^*$ on two symbols $\{0,1\}$ be defined by the forbidden words
\[
000, 110, 010, 101, 001, 111.
\]
Then the corresponding truthTable of the Follower Set Graph Algorithm is:

\smallskip

\begin{center}
\begin{tabular}{|c|c|c|c|c|c|c||l|}
\hline
11 & 01 & 10 & 00 & 1 & 0 & $\mt$ & \\
\hline
\hline
 &  &  &  &  &  &  & $\mt$ \\
\hline
$\qedsymbol$ & $\qedsymbol$ &  & $\qedsymbol$ &  &  &   & 0 \\
\hline
$\qedsymbol$ &  & $\qedsymbol$ & $\qedsymbol$ &  &  &  & 1 \\
\hline
$\qedsymbol$ & $\qedsymbol$ & $\qedsymbol$ & $\qedsymbol$ &  & $\qedsymbol$ & & 00 \\
\hline 
$\qedsymbol$ & $\qedsymbol$ & $\qedsymbol$ & $\qedsymbol$ & $\qedsymbol$ & $\qedsymbol$ &  & 10 \\
\hline 
$\qedsymbol$ & $\qedsymbol$ & $\qedsymbol$ & $\qedsymbol$ & $\qedsymbol$ & $\qedsymbol$ &  & 01 \\
\hline 
$\qedsymbol$ & $\qedsymbol$ & $\qedsymbol$ & $\qedsymbol$ & $\qedsymbol$ &  &  & 11 \\
\hline 
\end{tabular}
\end{center}

\smallskip

\noindent and gives the vertices
\begin{align*}
& w_0 = F(\mt), w_1 = F(0), w_2 = F(1), w_3 = F(00), w_4 = F(11) \\
& w_5 = F(10) = F(01) = F(100) = F(011).
\end{align*}
Then the follower set graph of $M^*$ is
{\small
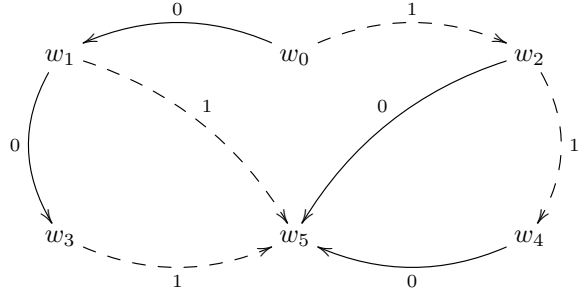
\[
\xymatrix@R=20mm@C=25mm{
w_1  \ar@{->}_{0}@/_1pc/[d] \ar@{-->}^{1}@/^1pc/[dr] & w_0 \ar@{-->}^{1}@/^1pc/[r] \ar@{->}_{0}@/_1pc/[l] & w_2 \ar@{-->}^{1}@/^1pc/[d] \ar@{->}_{0}@/_1pc/[dl] \\
w_3  \ar@{-->}_{1}@/_1pc/[r] & w_5 & w_4 \ar@{->}^{0}@/^1pc/[l]
}
\]
}

We see that that there are two unlabeled graph isomorphisms.
The first one sends $v_i$ to $w_i$ and the second one is the composition with the reflection along the vertical line that passes through $w_0$ and $w_5$.
Both of them do not lift to a labeled graph isomorphism as the path $11$ connecting $v_1$ with $v_5$ consists of two edges with the same label whereas its image $10$ has two edges of different labels. 

The languages $\La^*$ and $M^*$ do not arise from subshifts, but once more we can use their augmentations to produce a counterexample in this class.
We thus have
{\small
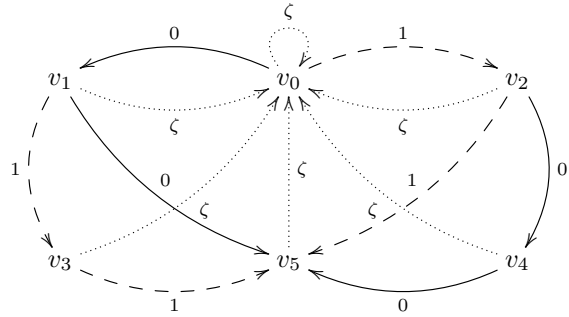
\[
\xymatrix@R=20mm@C=25mm{
v_1 \ar@{.>}_{\ze}@/_1pc/[r]  \ar@{-->}_{1}@/_1pc/[d] \ar@{->}^{0}@/_1pc/[dr] & v_0 \ar@{.>}^{\ze}@(ul,ur) \ar@{-->}^{1}@/^1pc/[r] \ar@{->}_{0}@/_1pc/[l] & v_2 \ar@{.>}^{\ze}@/^1pc/[l] \ar@{->}^{0}@/^1pc/[d] \ar@{-->}_{1}@/^1pc/[dl]\\
v_3 \ar@{.>}_{\ze}@/_1pc/[ur] \ar@{-->}_{1}@/_1pc/[r] & v_5 \ar@{.>}_{\ze}[u] & v_4 \ar@{.>}^{\ze}@/^1pc/[ul] \ar@{->}^{0}@/^1pc/[l]
}
\]
}for $\wt{\La}^*$, and
{\small
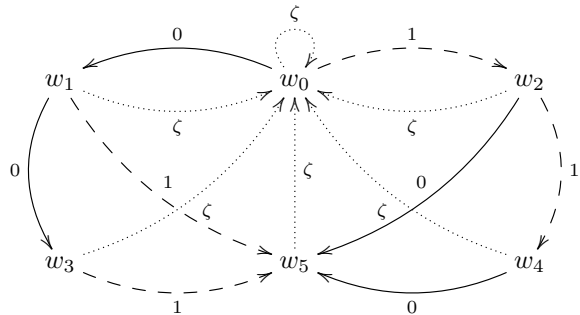
\[
\xymatrix@R=20mm@C=25mm{
w_1 \ar@{.>}_{\ze}@/_1pc/[r] \ar@{->}_{0}@/_1pc/[d] \ar@{-->}^{1}@/_1pc/[dr] & w_0 \ar@{.>}^{\ze}@(ul,ur) \ar@{-->}^{1}@/^1pc/[r] \ar@{->}_{0}@/_1pc/[l] & w_2 \ar@{.>}^{\ze}@/^1pc/[l] \ar@{-->}^{1}@/^1pc/[d] \ar@{->}_{0}@/^1pc/[dl] \\
w_3  \ar@{.>}_{\ze}@/_1pc/[ur] \ar@{-->}_{1}@/_1pc/[r] & w_5 \ar@{.>}_{\ze}[u] & w_4 \ar@{.>}^{\ze}@/^1pc/[ul] \ar@{->}^{0}@/^1pc/[l]
}
\]
}for $\wt{M}^*$.
Again the follower set graphs are not labeled graph isomorphic.
\end{example}

\subsection{Irreducible two-sided sofic subshifts}

Let us examine further the case of irreducible two-sided sofic subshifts.
Recall that a subshift is called \emph{irreducible} if it has a presentation through a labeled graph $\fG = (G, \fL)$ so that $G$ is irreducible.
Among all presentations, Fischer \cite{Fis75-1, Fis75-2} has shown that there exists a minimal resolving one, that is unique up to label graph isomorphism; see also \cite[Theorem 3.3.18]{LM95}.
Minimality is taken with respect to the number of vertices of the ambient graph.
This presentation is also known as the \emph{Fischer cover} of the subshift.
Fischer covers are follower-separated; for example see \cite[Corollary 3.319]{LM95}.
Uniqueness of the Fischer cover fails for reducible sofic subshifts; Jonoska \cite{Jon95} provides such a counterexample for subshifts of finite type.

The Fischer cover can be induced by the Krieger cover or by the follower set graph.
A way to obtain it from the follower set graph is as follows.
Recall that a word $\mu$ is called \emph{intrinsically synchronizing} if:
\begin{center}
whenever $\nu \mu \in \La^*$ and $\mu \nu' \in \La^*$ then $\nu \mu \nu' \in \La^*$.
\end{center}
Then the minimal resolving presentation is the labeled subgraph of the follower set graph formed by using just the follower sets of the intrinsically synchronizing words\footnote{\
It is worth mentioning that a subshift is of finite type if and only if all sufficiently long words are intrinsically synchronizing \cite[Exercise 3.3.5]{LM95}.
This property emphasizes the value of Jonoska's \cite{Jon95} counterexample.} (see \cite[Exercise 3.3.4]{LM95}).

We will see how the graph isomorphism we get from Theorem \ref{T:unlbl iso} induces an unlabeled graph isomorphism of the Fischer covers.
Moreover that it respects the vertices labeled by follower sets of intrinsically synchronizing words.
For the latter we will require some terminology and results from \cite{LM95}.
Recall that we follow the left version of their notation.

Given a labeled graph $\fG = (G, \fL)$ we say that a path $\mu$ is \emph{synchronizing} for $\fG$ if:
\begin{center} 
all paths labeled $\mu$ terminate at the same vertex.
\end{center}
Recall that we read paths from right to left.
Suppose in addition that $\fG$ is resolving and follower separated.
In this case  every path $w$ can be extended on the left to a synchronizing path $\mu = u w$ for $\fG$ by \cite[Proposition 3.3.16]{LM95}.
Under the same assumption if $\mu$ is synchronizing for $\fG$ then every path $u \mu$ is synchronizing for $\fG$ by \cite[Lemma 3.3.15]{LM95}.
The connection between synchronizing paths and intrinsically synchronizing words is given in \cite[Exercise 3.3.3]{LM95}.
That is, if $\fG$ is the minimal resolving presentation of a two-sided irreducible sofic shift $\La$ then a path $w$ is synchronizing for $\fG$ if and only if the word $w$ is intrinsically synchronizing for $\La$.

Now we have set the context for proving the next corollary.

\begin{corollary}\label{C: irr}
Let $\La$ and $M$ be two-sided irreducible sofic subshifts.
If $\La$ and $M$ are locally piecewise conjugate then there is an unlabeled graph isomorphism between their Fischer covers. 

Furthermore, the unlabeled graph isomorphism induces a bijection between the collections
\[
\{F_\La(\mu) \mid \textup{ $\mu$ is an intrinsically synchronizing word for $\La$}\}
\]
and
\[
\{F_M(\nu) \mid \textup{ $\nu$ is an intrinsically synchronizing word for $M$}\}.
\]
\end{corollary}

\begin{proof}
Let $\fG_\La = (G_\La, \L_1)$ and $\fG_M = (G_M, \L_2)$ be the follower set graphs of $\La$ and $M$, respectively.
Let $\fH_{\La}$ be the labeled graph that remains from $\fG_{\La}$ by using only the intrinsically synchronizing words, i.e. the Fischer cover of $\La$.
The graph isomorphism of Theorem \ref{T:unlbl iso} then gives an isomorphism of the ambient graph $H_\La$ of $\fH_\La$ onto a subgraph $H_M$ of $G_{M}$.
Let  $\fH_{M}$ be the labeled subgraph induced by $H_M$ inside $\fG_M$.
Notice that both $\fH_\La$ and $\fH_M$ are resolving as subgraphs of the resolving follower set graphs and thus
\[
h(X(\fH_\La)) = h(X(H_\La)) \qand h(X(\fH_M)) = h(X(H_M))
\]
by \cite[Proposition 4.13]{LM95}.

First we claim that $\fH_M$ gives a presentation of $M$.
If $X(\fH_M) \neq M$ then $X(\fH_M)$ is a proper subshift of $M$ and therefore  \cite[Corollary 4.4.9]{LM95} yields
\[
h(X(\fH_M)) < h(M).
\]
Since graph isomorphism of $H_\La$ with $H_M$ respects entropy we get that
\begin{align*}
h(\La) = h(X(\fH_\La)) = h(X(H_\La)) = h(X(H_M)) = h(X(\fH_M)) < h(M).
\end{align*}
However this contradicts $h(M) = h(\La)$ of Proposition \ref{P: same size}.
Hence $\fH_M$ is a presentation of $M$.
Notice here that the number of vertices of $\fH_M$ coincides with the minimal number of vertices required to describe $\La$, i.e.
\[
|(\fH_M)^{(0)}| = |(H_M)^{(0)}| = |(H_\La)^{(0)}| = |(\fH_\La)^{(0)}|.
\]

Secondly we claim that $\fH_M$ is minimal for $M$ with respect to the number of vertices.
Otherwise we could find a subgraph $\fH'_M$ of $\fG_M$ on less vertices than that of $\fH_M$.
Then the unlabeled graph isomorphism would carry over, as above, to a presentation $\fH'_\La$ of $\La$, giving
\begin{align*}
|(\fH'_\La)^{(0)}| = |(\fH'_M)^{(0)}| < |(\fH_M)^{(0)}| = |(\fH_\La)^{(0)}|.
\end{align*}
However this contradicts minimality of $\fH_\La$ for $\La$, and thus $\fH_M$ is a minimal presentation of $M$.

So far we have proved that $\fH_M$ is a minimal resolving presentation of $M$.
Initially $\fH_M$ is isomorphic to the subgraph of $\fG_M$ obtained by using the vertices labeled by follower sets of intrinsically synchronizing words.
We will show now that these two sub-graphs of $\fG_M$ are actually equal.
Due to minimality it suffices to show that the vertices of $\fH_M$ correspond to follower sets of intrinsically synchronizing words.

We will use that $\fH_M$ is follower-separated and irreducible as a minimal presentation of $M$ \cite[Corollary 3.3.19]{LM95}.
Let $F(w)$ be a vertex of $\fH_M$.
We will show that
\[
\textup{$\exists$ an intrinsically synchronizing word $\nu \in M^*$ such that $F(w) = F(\nu)$.}
\]
Consider a labeled path $u_1$ in $\fH_M$ starting at a vertex $J$ and extend it to a synchronizing path $u_2 u_1$ in $\fH_M$.
That is, all paths labeled $u_2 u_1$ end at the same vertex, say $J'$ in $\fH_M$.
By irreducibility of $\fH_M$ there is a path $u_3$ connecting $J'$ with $F(w)$ so that the path $u_3 u_2 u_1$ is allowable in $\fH_M$.
Moreover this is an extension of a synchronizing path and thus it is a synchronizing path in $\fH_M$, i.e. all paths in $\fH_M$ representing $u_3 u_2 u_1$ end at $F(w)$.
Minimality of $\fH_M$ implies that $u_3 u_2 u_1$ is an intrinsically synchronizing word for $M$. 
As $\fH_M$ represents $M$, the follower set of $u_3 u_2 u_1$ in $M$ coincides with the paths in $\fH_M$ starting at the vertex $F(w)$.
However the collection of these paths is exactly the set $F(w)$ due to the follower set graph construction.
Hence we conclude that $F(w) = F(u_3 u_2 u_1)$ and the proof is complete.
\end{proof}


\begin{remark}\label{R: intr mix}
Referring to the proof of Corollary \ref{C: irr}, we do not claim a direct bijection between intrinsically synchronizing words.
It is unclear whether the function of Proposition \ref{P: same size} respects this property.
The main obstacle is that the word
\[
f_n^{-1}(f_{|\nu|}(\nu) \mu f_{|\nu'|}(\nu')) \qfor n = |\nu| + |\mu| + |\nu'|
\]
gives a word $\nu'' \mu' \nu'$ instead of $\nu \mu' \nu'$ (as indicated in Example \ref{E: even vs sft}).
A similar obstacle does not allow checking directly whether the mixing property is also preserved under local piecewise conjugacy.
\end{remark}

\begin{example}\label{E: even vs sft irr}
We require both $\La$ and $M$ be irreducible in Corollary \ref{C: irr}.
This is because local piecewise conjugacy does not preserve irreducibility.
For an example recall the even shift and the subshift of finite type constructed in Example \ref{E: even vs sft}.
The even shift $\La$ is irreducible and the subgraph
{\small
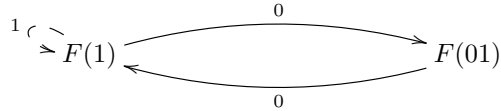
\[
\xymatrix@R=30pt{
F(1) \ar@{-->}_{1}@(ul,l) \ar@{->}^{0}@/^1pc/[rrrr] & & & & F(01) \ar@{->}^{0}@/^1pc/[llll]
}
\]
}gives its Fischer cover.
However the subshift $M$ on $\{001\}$ is not irreducible.
Any subgraph on less vertices produces a proper subshift of $M$.
In particular the unlabeled graph isomorphism from the Fischer cover of $\La$ to $M$ produces the irreducible subgraph
{\small
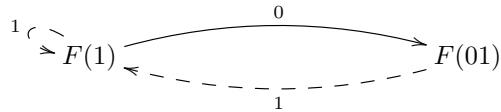
\[
\xymatrix@R=30pt{
F(1) \ar@{-->}_{1}@(ul,l) \ar@{->}^{0}@/^1pc/[rrrr] & & & & F(01) \ar@{-->}^{1}@/^1pc/[llll]
}
\]
}which does not represent $0^\infty \in M$.
\end{example}

\begin{example}
The converse of Corollary \ref{C: irr} does not hold.
That is, irreducible subshifts that are not locally piecewise conjugate may have Fischer covers that admit an unlabeled graph isomorphism.

For a counterexample let $M$ be defined by the forbidden word $\{00\}$.
Then the truthTable of $M$ is

\smallskip

\begin{center}
\begin{tabular}{|c|c|c|c|c|c|c||l|}
\hline
1 & 0 & $\mt$ & \\
\hline
\hline
 &  &  & $\mt$ \\
\hline
 & $\qedsymbol$  &  & 0 \\
\hline
 &  &  & 1 \\
\hline
\end{tabular}
\end{center}

\smallskip

\noindent and thus its follower set graph is given by
{\small
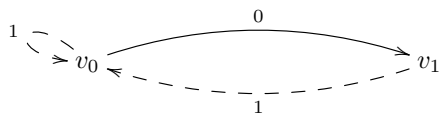
\[
\xymatrix@R=30pt{
v_0 \ar@{-->}_{1}@(ul,l) \ar@{->}^{0}@/^1pc/[rrrr] & & & & v_1 \ar@{-->}^{1}@/^1pc/[llll]
}
\]
}where $v_0 = F(\mt) = F(1)$ and $v_1 = F(0)$.
Consequently $M$ is irreducible (having one irreducible presentation) and its Fischer cover coincides with its follower set graph.
Comparing with the even shift $\La$ we see that the Fischer covers of $\La$ and $M$ admit an unlabeled graph isomorphism, but this does not hold for their follower set graphs.
\end{example}


\end{document}